\documentclass[11pt]{amsart}

\usepackage{amsmath, amsthm, amsfonts, amssymb, graphicx, mathrsfs, enumerate, subcaption}
\usepackage{color}
\usepackage[top=1.1in, bottom=1.1in, left=1.1in, right=1.1in]{geometry}
\usepackage{hyperref}
\usepackage{tikz}

\theoremstyle{definition}
\newtheorem{definition}{Definition}

\theoremstyle{plain}
\newtheorem{thm}{Theorem}[section]

\newtheorem{lemma}{Lemma}
\newtheorem{prop}{Proposition}[section]

\theoremstyle{remark}
\newtheorem{remark}{Remark}

\newcommand\cals{\mathcal{S}}
\newcommand\co{\colon\thinspace}

\begin{document}

\title{The maximal injectivity radius of hyperbolic surfaces with geodesic boundary}

\author{Jason DeBlois}
\address{Department of Mathematics, University of Pittsburgh}\email{jdeblois@pitt.edu}
\author{Kim Romanelli}
\address{Catalyte}\email{romanelli.ki@gmail.com}

\begin{abstract}
We give sharp upper bounds on the injectivity radii of complete hyperbolic surfaces of finite area with some geodesic boundary components.  The given bounds are over all such surfaces with any fixed topology; in particular, boundary lengths are not fixed.  This extends the first author's earlier result to the with-boundary setting.  In the second part of the paper we comment on another direction for extending this result, via the systole of loops function.\end{abstract}

\maketitle

The main results of this paper relate to maximal injectivity radius among hyperbolic surfaces with geodesic boundary. For a point $p$ in the interior of a hyperbolic surface $F$, by the \textit{injectivity radius of $F$ at $p$} we mean the supremum $\mathit{injrad}_p(F)$ of all $r>0$ such that there is a locally isometric embedding of an open metric neighborhood -- a disk -- of radius $r$ into $F$ that takes the disk's center to $p$. If $F$ is complete and without boundary then $\mathit{injrad}_p(F)=\frac{1}{2}\mathit{sys}_p(F)$ at $p$, where $\mathit{sys}_p(F)$ is the \textit{systole of loops at $p$}, the minimal length of a non-constant geodesic arc in $F$ with both endpoints at $p$. But if $F$ has boundary then $\mathit{injrad}_p(F)$ is bounded above by the distance from $p$ to the boundary, so it approaches $0$ as $p$ approaches $\partial F$ (and we extend it continuously to $\partial F$ as $0$). On the other hand, $\mathit{sys}_p(F)$ does not approach $0$ as $p\to\partial F$. One theme of this paper contrasts the behavior of these two functions among surfaces with boundary.

Our work belongs to a line of results going back at least to L.~Fejes T\'oth \cite{FejesToth}, who identified the supremal density of equal-radius circle packings of the hyperbolic plane $\mathbb{H}^2$. Later work of C.~Bavard \cite{Bavard} tied such packing density questions to the maximal injectivity radius of surfaces. For any closed surface $S$, the main results of \cite{Bavard} identify the maximum of $\mathit{injrad}_p(F)$ over all $p\in F$ and all hyperbolic surfaces $F$ homeomorphic to $S$, and they characterize such $F$ at which the maximum is attained in terms of a decomposition into hyperbolic polygons. Bavard observed that each maximizing closed surface arises as a quotient of $\mathbb{H}^2$ by the action of a group of isometries that stabilizes a maximal-density packing. The packing's disks have a fixed radius that depends only on the topology of $S$, and the surface group acts transitively on them. 

Schmutz \cite[Theorem 11]{Schmutz_maxRiem}, then later (and more generally) Bavard \cite{Bavard_anneaux}, extended these results to the context of hyperbolic surfaces with geodesic boundary using analogous techniques. Rather than addressing the injectivity radius at \textit{points} of the surface, these results bound what the English abstract of \cite{Bavard_anneaux} terms the \textit{injectivity radius of a disjoint union of closed geodesics}; a.k.a.~ the maximal collar width. (We thank the referee for pointing out these references to us.) A separate line of study, which includes recent works of Fanoni \cite{Fanoni} and Gendulphe \cite{Gendulphe_systole}, gives \textit{lower} bounds on maximal injectivity radius in various classes of hyperbolic orbifolds and manifolds.

Our first main result here gives upper bounds on the injectivity radius at points of surfaces with boundary, in the sense of the first paragraph above. Its proof uses techniques that the first author introduced a few years ago in extending the bound of \cite{Bavard} to finite-area, non-closed surfaces \cite{DeB_Voronoi}. Unlike in the closed case, for a maximal-radius disk in a non-compact surface the preimage circle packing of $\mathbb{H}^2$ does not have maximal density for its radius.  So we took a different, though related, approach in \cite{DeB_Voronoi} which extends to the bounded case with a few adaptations.

\theoremstyle{plain}
\newtheorem*{maxinjradthrm}{Theorem \ref{maximal injrad}}
\begin{maxinjradthrm} Fix natural numbers $\chi<0$ and $b,n\ge 0$. For a complete, finite-area hyperbolic surface $F$ with compact geodesic boundary that has $n$ cusps, $b$ boundary components and Euler characteristic $\chi$, and any $p\in F$, $\mathit{injrad}_p(F) \leq r_{\chi, n, b}$, where $r_{\chi, n, b}$ is the unique solution to:
$$ 3\left(2-\left(2\chi+b+n\right) \right) \alpha(r_{\chi, n, b}) +2n\beta(r_{\chi, n, b}) + 2b \gamma(r_{\chi, n, b}) = 2\pi. $$
Here, $\alpha(r) = 2\sin^{-1}{\left(\frac{1}{2\cosh{r}}\right)}$,  
	$\beta(r) = \sin^{-1}{\left(\frac{1}{\cosh{r}}\right)}$, and 
	$\gamma(r) = 2\sin^{-1}{\left(\frac{1}{\sqrt{2}\cosh{r}} \right)}$
	each measure vertex angles, respectively that of an equilateral triangle with side length $2r$, 
	at finite vertices of a horocyclic ideal triangle with compact side length $2r$, 
	and of a square with side length $2r$. 
	
	This bound is sharp. It is attained at finitely many surfaces up to isometry, which all decompose into a common collection of hyperbolic polygons.\end{maxinjradthrm}
	
Theorem \ref{maximal injrad} is stated in compressed form here in the introduction. Its full statement, in Section \ref{packrad}, describes the decomposition of maximizing surfaces into hyperbolic polygons in detail. 

The bound of Theorem \ref{maximal injrad} differs from the main result of \cite{DeB_Voronoi} in the appearance of the angle measure $\gamma$ in the equation defining the bound. This in turn reflects the appearance of \textit{Saccheri quadrilaterals}, which are obtained by cutting squares in half along perpendicular bisectors between opposite sides, in the decomposition of maximizing surfaces. The \textit{base} of such a quadrilateral, its side in the perpendicular bisector, wraps around each geodesic boundary component of a maximizing surface with boundary. Above, a \textit{horocyclic ideal triangle} has two vertices on a horocycle $C$ and a single ideal vertex at the ideal point of $C$.

We prove Theorem \ref{maximal injrad} in Section \ref{packrad} using the ``centered dual decomposition'' defined in \cite{DeB_Voronoi}.  For a surface $F$ with geodesic boundary and a finite set $\cals\subset\mathit{int}\,F$, we double $F$ across its boundary to produce a surface $\mathit{DF}$ without boundary and a finite set $\cals\cup\overline\cals$ invariant under the reflection of $\mathit{DF}$ exchanging $F$ with its mirror image. We show that the centered dual determined by $\cals\cup\overline\cals$ is also reflection-invariant and, in Proposition \ref{maximal rad}, adapt the arguments of \cite{DeB_Voronoi} to this setting to bound the maximal radius of a packing of $\mathit{DF}$ by disks centered in $\cals\cup\overline\cals$. We complete the proof of Theorem \ref{maximal injrad} by showing that this bound is attained when $\cals = \{p\}$ is a singleton.

The next two sections consider the systole of loops function, defined as in the first paragraph, which offers a different direction for extending \cite[Thm.~5.11]{DeB_Voronoi} to the bounded context.  As mentioned above, if $\partial F = \emptyset$ then $\mathit{sys}_p(F) = 2\mathit{injrad}_p(F)$ for each $p\in F$, since for a maximal-radius disk based at $p$, a minimal-length geodesic arc based at $p$ consists of two radial arcs that meet at a point of self-tangency.  But if $F$ has boundary then 
\[ \mathit{injrad}_p(F) = \min\left\{ \frac{1}{2}\mathit{sys}_p(F), \mathit{dist}(p,\partial F) \right\} \]
is strictly less than $\mathit{sys}_p(F)$ in general, for instance near $\partial F$.  

In fact as we show in Section \ref{sysloops}, the problem of maximizing $\mathit{sys}_p(F)$ fundamentally differs from that of maximizing $\mathit{injrad}_p(F)$, in the following sense: by Proposition \ref{up up up}, for a hyperbolic surface $F$ with geodesic boundary, the maximum of $\mathit{sys}_p(F)$ over all $p\in F$ is bounded below by a linear function of the length of the longest component of $\partial F$.  So since boundary lengths can go to infinity the systole of loops function has no global maximum over all bounded hyperbolic surfaces with a fixed topology; that is, there is no direct analog of Theorem \ref{maximal injrad} concerning $\mathit{sys}_p$. 

Instead, one can fix both a topology and a collection of boundary lengths, and seek to maximize $\mathit{sys}_p(F)$ over all points $p$ in all hyperbolic surfaces $F$ with geodesic boundary that have the given topology \textit{and boundary lengths}. The preprint \cite{Gendulphe} takes this perspective. Theorem 1.2 there describes an equation determined by this data and claims that the maximum value is the unique positive solution of this equation. However for certain collections of boundary lengths, in particular when one is much longer than any other, the equation given in \cite[Theorem 1.2]{Gendulphe} has no solutions.  We prove this in Lemma \ref{F} and expand on the issue in Section \ref{three hole}.  

We suspect that the formula given in \cite[Theorem 1.2]{Gendulphe} requires adjustment only for boundary length collections with exactly one member that is ``unreasonably large''.  We confirm this and make the adjustment for the simplest case, that of the three-holed sphere, in Theorem \ref{maximal sysloops}. This case is simplest because the three-holed sphere is \textit{rigid}: there is a unique such sphere $F$, up to isometry, with any given collection $\{b_1,b_2,b_3\}\subset (0,\infty)$ of boundary lengths. So to maximize $\mathit{sys}_p(F)$, we must just optimize the choice of $p\in F$. Even this task is somewhat subtle.

\theoremstyle{plain}
\newtheorem*{sysloopsthrm}{Theorem \ref{maximal sysloops}}
\newcommand\MaxSysLoops{ 
	Let $F$ be a hyperbolic three-holed sphere with geodesic boundary components $B_1$, $B_2$, $B_3$ of respective lengths $b_1 \leq  b_2 \leq b_3 \in (0, \infty)$.
	\begin{enumerate}[a)]
	\item The maximum value of $\mathit{sys}_p(F)$, taken over all $p\in F$, is attained in the interior of $F$ if and only if $f_{(b_1,b_2)}(b_3) > \pi$, where
	 \[ f_{(b_1,b_2)}(x) = 6\sin^{-1}\left(\frac{1}{2\cosh(x/2)}\right) + 2\sum\limits_{i=1}^{2}\sin^{-1}\left(\frac{\cosh(b_i /2)}{\cosh(x/2)}\right).\] 
	 In this case it is the unique solution $x\in(b_3,\infty)$ to $f_{(b_1,b_2)}(x) + \sin^{-1}\left(\frac{\cosh(b_3/2)}{\cosh(x/2)}\right) =2\pi$.
	For a point $p\in\mathit{int}(F)$ at which the maximum value is attained, the systolic loops based at $p$ divide $F$ into an equilateral triangle and three one-holed monogons.
	\item The maximum value of $\mathit{sys}_p(F)$ is attained on $\partial F$ if and only if $f_{(b_1,b_2)}(b_3) \leq \pi$. In this case $b_3$ is strictly larger than $b_1$ and $b_2$, and the maximum is attained on the component $B_3$ of length $b_3$. The maximum value of $\mathit{sys}_p(F)$ is the unique $x \in (b_2,b_3]$ such that $g(x) = \pi$, where
	\begin{align*}
	g(x) = 2\cos^{-1}{\left( \frac{\tanh{(b_3 /2)}}{\tanh{(x)}}\right)} + 2\sin^{-1}{\left(\frac{\sinh{(b_3 /2)}}{\sinh{(x)}}\right)} + 2\sum\limits_{i=1}^{2}\sin^{-1}{\left(\frac{\cosh{(b_i /2)}}{\cosh{(x/2)}}\right)}, \end{align*}
and for $p\in F$ where the maximum is attained, the systolic loops based at $p$ divide $F$ into an isosceles triangle and two one-holed monogons.
	\end{enumerate}
}
\begin{sysloopsthrm}\MaxSysLoops\end{sysloopsthrm}

Above, a \textit{one-holed monogon} refers to the quotient space of a Saccheri quadrilateral $Q$ obtained by isometrically identifying its sides that share endpoints with the base, so that the base's endpoints are paired. This space is homeomorphic to an annulus. Away from the summit vertex quotient (where the total angle is less than $\pi$) it has the structure of a hyperbolic surface with geodesic boundary, with the quotient of the base of $Q$ as a geodesic boundary component.

The equation determining the bound of Theorem \ref{maximal sysloops}(a) matches the three-holed sphere case of that given in \cite[Theorem 1.2]{Gendulphe}; the equation $g(x)=\pi$ determining the bound of case (b) has no correspondent in \cite{Gendulphe}. It is distinguished from the first by the fact that the maximum given by its solution is at most the largest boundary length $b_3$. 

\subsection*{Acknowledgements} Sections \ref{packrad} and \ref{three hole} of this paper are adapted from the second author's 2018 University of Pittsburgh PhD thesis, directed by the first author. We thank thesis committee members Tom Hales, Chris Lennard, and Matt Stover for helpful feedback. We are also grateful to the referee for helpful comments and references.

\section{Bounds on packing radius}\label{packrad}

In this section we prove Theorem \ref{maximal rad}. Its upper bound follows from the more general Proposition \ref{maximal rad} below, which gives an upper bound $r^k_{\chi,n,b}$ on the radius of packings of hyperbolic surfaces with boundary by $k$ equal-radius disks, for arbitrary $k\in\mathbb{N}$.  This result extends Proposition 0.2 of \cite{DeB_many}, itself a follow-up to \cite{DeB_Voronoi}, to the with-boundary setting.  Here, given a hyperbolic surface $F$ with geodesic boundary, we double it across $\partial F$ to produce a boundaryless hyperbolic surface $\mathit{DF}$.  Then for a finite set $\cals\subset F$, we analyze the centered dual complex of $\cals\cup\overline{\cals}$, where $\overline\cals\subset \mathit{DF}$ is the reflection of $\cals$ across $\partial F$.  This complex is itself reflection-invariant, by Lemma \ref{Voronoi preserved} below.

Before proceeding to the proofs, we give a few definitions. A \textit{hyperbolic half-plane} is the closure of a component of $\mathbb{H}^2-\gamma$, where $\gamma$ is a geodesic in the hyperbolic plane $\mathbb{H}^2$.  Its \textit{boundary} is the geodesic $\gamma$.  A \textit{hyperbolic surface with geodesic boundary} is a Hausdorff, second-countable topological space with an atlas of charts to hyperbolic half-planes, such that transition functions are restrictions of hyperbolic isometries.  For such a surface $F$, the \textit{boundary} $\partial F$ is the set of points mapped into half-plane boundaries by chart maps, and the \textit{interior} is $\mathit{int}\,F = F-\partial F$.  $F$ is \textit{complete} if its universal cover is identified (via a developing map) with the intersection of a countable collection of hyperbolic half-planes with pairwise disjoint boundaries.

\begin{definition}
For a hyperbolic surface $F$ with geodesic boundary, the \textit{double} of $F$ is the quotient space $\mathit{DF} = F \cup \overline{F} / \sim$, where $\overline{F}$ is a second copy of $F$ and $x\sim y$ if and only if $y=x$ or $x\in\partial F$ and $y=\bar{x}$ is the point of $\overline F$ corresponding to $x$. The \textit{reflection across $\partial F$} is the map $r\co\mathit{DF}\to\mathit{DF}$ that exchanges $x$ and $\bar{x}$ for each $x\in F$.\end{definition}

We give $\mathit{DF}$ the structure of a hyperbolic surface (without boundary) as follows.  For $x\in\mathit{int}\,F$ and a chart map $\phi\co U\to\mathbb{H}^2$, where $U\subset F$ is open, let $U_0=U\cap\mathit{int}\,F$ and let $\phi_0$ be the restriction of $\phi$ to $U_0$.  For the corresponding point $\bar{x}\in\overline{F}$, let $\overline{U}_0 = r(U_0)$ and $\bar\phi_0 = \rho\circ\phi_0\circ r$, where $\rho\co\mathbb{H}^2\to\mathbb{H}^2$ is the reflection through the boundary geodesic of the half-plane to which $\phi$ maps.  For $x\in \partial F$ we take $U_0 = U\cup r(U)$ and define $\phi_0$ as $\phi$ on $U$ and $\rho\circ\phi\circ r$ on $r(U)$.

It is an exercise to show that these chart maps have isometric transition functions and that the reflection $r$ across $\partial F$ is an isometry of the resulting hyperbolic structure.  It is also a fact that $\mathit{DF}$ is complete if and only if $F$ is complete.  In the complete case there is a locally isometric universal cover $\pi\co\mathbb{H}^2\to \mathit{DF}$, and the restriction of this cover to any component of the preimage of $F$ is its universal cover.


Our first few results describe aspects of the structure of the Voronoi tessellation of $\mathit{DF}$ determined by $\cals\cup\overline\cals$, for a locally finite set $\cals\subset F$. To begin, we recall the general definition of the Voronoi tessellation of a complete hyperbolic surface $F$ determined by a locally finite set $\cals$. (Compare with \cite[\S 1]{DeB_Voronoi}.)  Fix a locally isometric universal cover $\pi\co\mathbb{H}^2\to F$ and take $\widetilde\cals = \pi^{-1}\cals$. For each $s\in\widetilde\cals$ we define an associated Voronoi two-cell $V_s$ as
\[ V_s = \{ x\in\mathbb{H}^2\,|\, d(s,x)\le d(t,x)\ \forall\ t\in\widetilde\cals\,\}. \]
Each $V_s$ so-defined is a hyperbolic polygon; for $t\ne s$, $V_t$ intersects $V_s$ only at an edge or vertex (if at all); and the union  of all such $V_s$ covers $\mathbb{H}^2$. The \textit{Voronoi tessellation of $\mathbb{H}^2$ determined by $\widetilde\cals$} is the locally finite cell complex with two-cells the $V_s$ for $s\in\widetilde\cals$, one-cells the edges of such $V_s$, and zero-cells their vertices.

Each Voronoi edge is of the form $e = V_s\cap V_t$, where $s,t\in\widetilde{\cals}$ determine Voronoi two-cells $V_s$ and $V_t$, and we define the \textit{geometric dual} of $e$ to be the geodesic arc $\gamma_{st}$ joining $s$ to $t$. Following \cite[Dfn.~2.1]{DeB_Voronoi}, we say a Voronoi edge $e$ is \textit{non-centered} if $e$ does not intersect its geometric dual in $\mathit{int}\,e$ (see the figure below) and define the \textit{non-centered Voronoi subgraph} determined by $\widetilde{\cals}$ as the set of non-centered Voronoi edges, together with their endpoints.

\begin{figure}[ht]
\begin{tikzpicture}

\begin{scope}[xshift=-2.2in]

\node [above] at (-2.7,0) {\large Centered};
\node [below] at (-2.7,0) {\large Case};

\draw [very thick] (-0.577,-1) -- (0,0) -- (1.1547,0);
\draw [very thick] (-0.577,1) -- (0,0);
\fill [color=black] (0.5,-0.866) circle [radius=0.05];
\fill [color=black] (0.5,0.866) circle [radius=0.05];
\fill [color=black] (-1,0) circle [radius=0.05];

\node [right] at (0.5,-0.866) {\small $s$};
\node [right] at (0.5,0.866) {\small $t$};
\node [above] at (1,0) {\small $e$};

\draw [thin] (0.5,-0.866) -- (0.5,0.866);

\end{scope}

\begin{scope}[xshift=1in]

\node [above] at (-2.9,0) {\large Non-centered};
\node [below] at (-2.7,0) {\large Case};

\draw [very thick] (-0.8512,-0.7802) -- (0,0) -- (1.1547,0);
\draw [very thick] (-0.8512,0.7802) -- (0,0);
\fill [color=black] (-0.087,-0.996) circle [radius=0.05];
\fill [color=black] (-0.087,0.996) circle [radius=0.05];
\fill [color=black] (-1,0) circle [radius=0.05];

\node [right] at (-0.087,-0.996) {\small $s$};
\node [right] at (-0.087,0.996) {\small $t$};
\node [above] at (1,0) {\small $e$};

\draw [thin] (-0.087,-0.996) -- (-0.087,0.996);

\end{scope}

\end{tikzpicture}
\caption{\label{big fur}A Voronoi edge $e = V_s\cap V_t$ and its geometric dual $\gamma_{st}$, when $e$ is centered (left) vs.~non-centered (right). (Voronoi edges in bold; $\gamma_{st}$ drawn thin.)}
\end{figure}
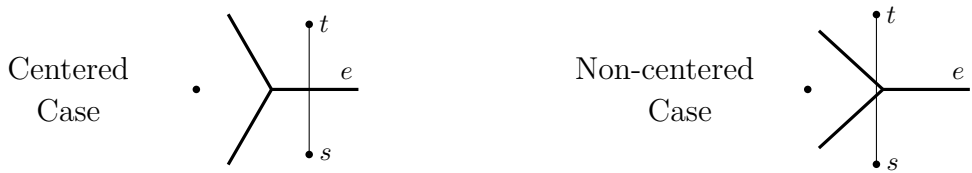

Being canonically determined by the geometry of $\widetilde\cals$, the Voronoi tessellation is invariant under isometries of $\mathbb{H}^2$ that preserve $\widetilde\cals$. The same holds true for the non-centered Voronoi subgraph, cf.~\cite[Lemma 5.4]{DeB_Voronoi}. In particular, they are invariant under the isometric $\pi_1 F$-action by covering transformations, so they project respectively to the \textit{Voronoi tessellation of $F$ determined by $\cals$} --- a locally finite cell decomposition of $F$ ---  and its \textit{non-centered Voronoi subgraph}.

The role of the complete hyperbolic surface $F$ of this definition is played in the current setting by the double $\mathit{DF} = F\cup\overline{F}$ of a hyperbolic surface $F$ with geodesic boundary, and the role of $\cals$ is played by $\cals\cup\overline\cals$ for a locally finite subset $\cals$ of $F$, where $\overline\cals$ is the mirror image of $\cals$ in the mirror image $\overline{F}$ of $F$. We have:

\begin{lemma}\label{Voronoi preserved} For a complete hyperbolic surface $F$ with geodesic boundary and a finite set $\mathcal{S} \subset F$, the Voronoi tessellation of $\mathit{DF}$ determined by $\cals\cup\overline\cals$ is preserved by the reflection through $\partial F$, where $\overline\cals = \{\bar{s}\,|\,s\in\cals\}$.  This reflection also preserves the non-centered Voronoi subgraph.\end{lemma}

\begin{proof} Taking $\pi\co\mathbb{H}^2\to\mathit{DF}$ to be a locally isometric universal cover and $\widetilde\cals = \pi^{-1}(\cals\cup\overline\cals)$, this again follows from the observation recorded above the Lemma, that the Voronoi tessellation determined by $\widetilde\cals$ is invariant under isometries of $\mathbb{H}^2$ that preserve it. If $r\co\mathit{DF}\to\mathit{DF}$ is the reflection fixing $\partial F$ and exchanging $F$ with $\overline{F}$, this holds for any lift $\tilde{r}$ of $r$ to $\mathbb{H}^2$. The Voronoi tessellation of $\mathbb{H}^2$ determined by $\widetilde\cals$ is thus $\tilde{r}$-invariant, so its projection to $\mathit{DF}$ is $r$-invariant. The same argument gives the same assertion for the non-centered Voronoi subgraph.
\end{proof}

\begin{lemma}\label{boundary is Voronoi}  For a complete hyperbolic surface $F$ with geodesic boundary and a locally finite set $\mathcal{S} \subset \mathit{int}\,F$, $\partial F$ lies in the Voronoi graph (i.e. the one-skeleton of the Voronoi tessellation) of the double $\mathit{DF}$ of $F$ determined by $\mathcal{S} \cup \overline{\mathcal{S}}$, where $\overline{\cals} = \{\bar{s}\,|\,s\in \cals\}$. 
\end{lemma}

\begin{proof}  For a set $\widetilde{\cals}\subset \mathbb{H}^2$, the Voronoi graph of $\widetilde\cals$ is characterized as the set of $x\in\mathbb{H}^2$ that have at least two closest points in $\widetilde\cals$.  Let us now take $\widetilde\cals = \pi^{-1} (\cals\cup\overline\cals)$, where $\pi\co\mathbb{H}^2\to \mathit{DF}$ is a locally isometric universal cover.  For any geodesic $\gamma$ in the preimage of $\partial F$ under $\pi$, the reflection $r\co \mathit{DF}\to \mathit{DF}$ across $\partial F$ lifts to a reflection $\tilde{r}$ of $\mathbb{H}^2$ fixing $\gamma$.  Since $\cals\cup\overline\cals$ is invariant under $r$, $\widetilde{\cals}$ is invariant under $\tilde{r}$, so for any $x\in \gamma$ and closest point $s\in\widetilde{\cals}$, $\tilde{r}(s)\in\widetilde{\cals}$ is also a closest point to $x$.  And $s\neq\tilde{r}(s)$ since $\cals\subset\mathit{int}\,F$, so $x$ lies in the Voronoi graph of $\widetilde\cals$.\end{proof}

Before stating the next result we recall a few more results and definitions from \cite{DeB_Voronoi}. By Lemma 1.5 and Fact 1.6 there, for each vertex $v$ of the Voronoi tessellation determined by a locally finite set $\widetilde\cals\subset \mathbb{H}^2$ there exists $J_v>0$, the \textit{radius} of $v$, such that $d(v,s) = J_v$ for each $s\in\widetilde\cals$ such that $V_s$ contains $v$ and $d(v,t) > J_v$ for all other $t\in\widetilde\cals$. Vertex radius is preserved by isometries of $\mathbb{H}^2$ that preserve $\widetilde\cals$, so if $\widetilde\cals$ is the preimage of a set $\cals\subset F$ under a locally isometric universal cover from $\mathbb{H}^2$ to a complete hyperbolic surface $F$ then each vertex of the Voronoi tessellation of $F$ determined by $\cals$ inherits a well-defined radius, cf.~\cite[Lemma 5.4]{DeB_Voronoi}.

Now for a complete hyperbolic surface $F$ with compact geodesic boundary and a finite set $\cals\subset F$, fix a locally isometric universal cover $\pi\co\mathbb{H}^2\to \mathit{DF}$ and let $\widetilde\cals = \pi^{-1}(\cals\cup\overline\cals)$. Lemma 5.5 of \cite{DeB_Voronoi} implies that each component $T$ of the non-centered Voronoi graph of $\mathbb{H}^2$ determined by $\widetilde\cals$ is a tree, with finite vertex set, that embeds in $\mathit{DF}$. For such a component $T$ let $v_T$ be a vertex of maximal radius among all (of the finitely many) vertices of $T$. Following Definition 2.8 of \cite{DeB_Voronoi} we call this the \textit{root vertex} of $T$. As remarked immediately below that definition, by \cite[Lemma 2.7]{DeB_Voronoi} $v_T$ is the unique vertx of $T$ with maximal radius.

\begin{lemma}\label{symmetric trees} Suppose $F$ is a complete hyperbolic surface with compact geodesic boundary, and $\cals\subset F$ is finite, and let $T$ be a component of the non-centered Voronoi subgraph of $\mathit{DF}$ determined by $\cals\cup\overline\cals$. If $T$ intersects the union of the boundary geodesics then it is taken to itself by reflection across $\partial F$, it is compact, and its root vertex is in the union of the boundary geodesics.
\end{lemma}

	\begin{proof} Recall from Lemma \ref{Voronoi preserved} that the non-centered Voronoi subgraph is preserved by reflection across $\partial F$, so $T$ is taken to another component of the non-centered Voronoi subgraph.  But since it intersects $\partial F$, which is fixed by the reflection, it intersects its image and hence is preserved (being a component).

By Proposition 2.9 of \cite{DeB_Voronoi}, $T$ can have at most one non-compact edge, so if it were non-compact then that edge would lie in the fixed locus $\partial F$ of the reflection. This cannot be, since $\partial F$ is compact but each end of a non-compact Voronoi edge exits a cusp.  Finally, since the root vertex of $T$ is unique with maximal radius it must also lie in the fixed locus $\partial F$ of the reflection, which can be easily showed to preserve vertex radius.\end{proof}

We have already defined the geometric dual to a Voronoi edge above. Now for a vertex $v$ of the Voronoi tessellation determined by $\widetilde\cals\subset\mathbb{H}^2$, we define the geometric dual to $v$ to be the convex hull $C_v$ of the set of $s\in\widetilde\cals$ such that $v\in V_s$, cf.~\cite[Prop.~1.1]{DeB_Voronoi}. All such $s$ lie on a circle of radius $J_v$ centered at $v$, the \textit{circumcircle} of $C_v$, see Figure \ref{pig cure}. $C_v$ is a compact, convex polygon with its vertex set in $\widetilde\cals$, see \cite[Lemma 1.5]{DeB_Voronoi}.  Finally, for $s\in\widetilde\cals$ we take the geometric dual to the Voronoi two-cell $V_s$ to be $s$ itself.  The collection of geometric dual cells is a polyhedral complex (see eg. \cite[Theorem 1.2]{DeB_Voronoi}) that we call the \textit{geometric dual complex} determined by $\widetilde\cals$.  

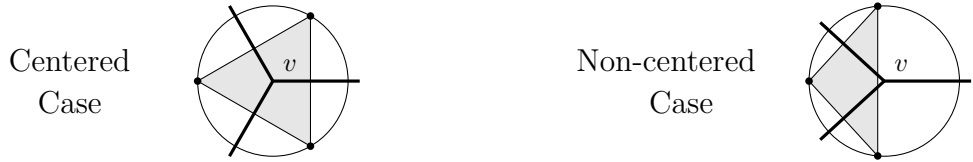
\begin{figure}[ht]
\begin{tikzpicture}

\begin{scope}[xshift=-2.2in]

\node [above] at (-2.7,0) {\large Centered};
\node [below] at (-2.7,0) {\large Case};

\draw [very thick] (-0.577,-1) -- (0,0) -- (1.1547,0);
\draw [very thick] (-0.577,1) -- (0,0);
\fill [color=black] (0.5,-0.866) circle [radius=0.05];
\fill [color=black] (0.5,0.866) circle [radius=0.05];
\fill [color=black] (-1,0) circle [radius=0.05];

\node [above right] at (0,0) {\small $v$};

\fill [opacity=0.1] (-1,0) -- (0.5,-0.866) -- (0.5,0.866);
\draw [thin] (-1,0) -- (0.5,-0.866) -- (0.5,0.866) -- cycle;
\draw (0,0) circle [radius=1];

\end{scope}

\begin{scope}[xshift=1in]

\node [above] at (-2.9,0) {\large Non-centered};
\node [below] at (-2.7,0) {\large Case};

\draw [very thick] (-0.8512,-0.7802) -- (0,0) -- (1.1547,0);
\draw [very thick] (-0.8512,0.7802) -- (0,0);
\fill [color=black] (-0.087,-0.996) circle [radius=0.05];
\fill [color=black] (-0.087,0.996) circle [radius=0.05];
\fill [color=black] (-1,0) circle [radius=0.05];

\node [above right] at (0,0) {\small $v$};

\fill [opacity=0.1] (-1,0) -- (-0.087,-0.996) -- (-0.087,0.996);
\draw [thin] (-1,0) -- (-0.087,-0.996) -- (-0.087,0.996) -- cycle;
\draw (0,0) circle [radius=1];

\end{scope}

\end{tikzpicture}
\caption{The geometric dual, pictured shaded, to a vertex $v$ of the Voronoi edge $e$ in the two cases of Figure \ref{big fur}. The circumcircle is also drawn in each case.}
\label{pig cure}
\end{figure}

If $\widetilde\cals = \pi^{-1}(\cals)$, where $\pi$ is a locally finite universal cover from $\mathbb{H}^2$ to a complete hyperbolic surface $F$ and $\cals\subset F$ is locally finite, then $\pi$ embeds the interior of each geometric dual cell in $F$ (see Remark 5.3 of \cite{DeB_Voronoi}).  We call the \textit{geometric dual complex of $F$ determined by $\cals$} the projection of the geometric dual complex of $\widetilde\cals$, and below we will count the ``vertices'' of a geometric dual two-cell in $F$ by those of a cell in $\mathbb{H}^2$ projecting to it.

\begin{lemma}\label{even vertices}  For a complete hyperbolic surface $F$ and a finite set $\cals\subset\mathit{int}\,F$, the  geometric dual complex of $\mathit{DF}$ determined by $\mathcal{S} \cup \overline{\mathcal{S}}$ is invariant under the reflection across $\partial F$, and the geometric dual to each Voronoi vertex in $\partial F$ has an even number of vertices.
\end{lemma}

\begin{proof}  Let $\widetilde\cals = \pi^{-1}(\cals\cup\overline\cals)$, for a locally finite universal cover $\pi\co\mathbb{H}^2\to \mathit{DF}$.  Since the geometric dual complex of $\widetilde\cals$ is defined geometrically it is invariant under \textit{all} isometries that preserve $\widetilde\cals$, and as in previous results this implies that the geometric dual complex of $\mathit{DF}$ determined by $\cals\cup\overline\cals$ is invariant under reflection across $\partial F$.

Now suppose $C=\pi(\widetilde{C})$ is a two-cell of the geometric dual complex of $\cals\cup\overline\cals$ that is dual to a Voronoi vertex $v\in\partial F$.  Then $\widetilde{C}$ is dual to a Voronoi vertex $\tilde{v}$ in a component $\gamma$ of $\pi^{-1}(\partial F)$, so since $\tilde{v}$ is invariant under the reflection through $\gamma$ --- which preserves $\widetilde{\cals}$ since it is a lift of the reflection across $\partial F$ --- $\widetilde{C}$ is also invariant under this reflection.  The reflection therefore preserves the vertex set of $\widetilde{C}$, and since $\cals\subset\mathit{int}\,F$ no vertex is fixed.  So the vertices come in pairs exchanged by the reflection, hence there is an even number.
\end{proof}

The geometric dual complex of a complete hyperbolic surface $F$ determined by a locally finite set $\cals$ is a subcomplex of the \textit{Delaunay tessellation} of $F$ determined by $\cals$.  In addition to the geometric dual cells, the Delaunay tessellation has one non-compact two-cell enclosing a horocyclic neighborhood of each cusp of $F$.  Each such ``horocyclic'' two-cell is the projection of a cell in the universal cover that is the convex hull of $\widetilde{\cals}\cap B$, where $\widetilde{\cals}$ is the preimage of $\cals$ and $B$ is a horoball bounded by a horocycle $S$ with the property that $S\cap\widetilde{\cals} = B\cap\widetilde{\cals}$. We may divide such cells into horocyclic ideal triangles using a collection of geodesic rays, one from each point of $B\cap\widetilde{\cals}$ to the ideal point of $B$. The Delaunay tessellation and its relation to the geometric dual complex are described in \cite{DeB_Delaunay}.

The \textit{centered dual decomposition}, constructed in \cite[\S 2]{DeB_Voronoi}, has two-cells of two forms. For each component $T$ of the non-centered Voronoi subgraph, the union of geometric dual two-cells dual to vertices of $T$, together with one of the horocyclic ideal triangles above if $T$ has a non-compact edge, is one type of cell. The simplest such example is pictured in Figure \ref{big sur}. The second type are those geometric dual two-cells not dual to a vertex of the non-centered Voronoi subgraph.

\begin{figure}[ht]
\begin{tikzpicture}

\begin{scope}[xshift=1in]

\draw [very thick, dotted] (-0.8512,0.7802) -- (0,0) -- (-0.8512,-0.7802);
\draw [very thick] (0,0) -- (0.8,0); 
\draw [very thick, dotted]  (1.245,-1) -- (0.8,0) -- (1.245,1);
\fill [color=black] (0,0) circle [radius=0.04];
\fill [color=black] (0.8,0) circle [radius=0.04];
\draw (0.8,0) circle [radius=0.07];

\fill [color=black] (-0.087,-0.996) circle [radius=0.05];
\fill [color=black] (-0.087,0.996) circle [radius=0.05];
\fill [color=black] (-1,0) circle [radius=0.05];
\fill [color=black] (2.133,0) circle [radius=0.05];

\node [above] at (0.4,0) {\small $e$};
\node [right] at (0.8,0) {\small $v_T$};

\fill [opacity=0.1] (-0.087,-0.996) -- (-1,0) -- (-0.087,0.996) -- (-0.087,-0.996) -- (2.133,0) -- (-0.087,0.996);
\draw [thin] (-0.087,-0.996) -- (-1,0) -- (-0.087,0.996) -- (-0.087,-0.996) -- (2.133,0) -- (-0.087,0.996);

\end{scope}

\end{tikzpicture}
\caption{\label{big sur} A component $T$ of the non-centered Voronoi subgraph with the single edge $e$ and root vertex $v_T$ (circled). Other Voronoi edges are dotted. The centered dual two-cell associated to $T$ is the union of the shaded triangles.}
\end{figure}
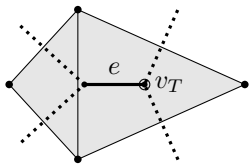

The great advantage of the centered dual complex is that the area of a two-cell is bounded below in terms only of its combinatorics and a lower bound on its edge lengths. The result below records such a bound that is relevant to the case under consideration in this paper.

\begin{lemma}\label{new bound} 
	Let $C$ be a compact 2-cell of the centered dual complex of a complete, finite-area hyperbolic surface $F$ which intersects the union of boundary geodesics, and suppose that for fixed $d>0$ each edge of $\partial C$ has length at least $d$.
	If $C$ is a quadrilateral, then its area is at least $D_0(d,d,d,d)$, that of a hyperbolic square with side lengths $d$.
	If $\partial C$ has $k>4$ edges, then Area$(C) \geq (k-4)A_m(d) + D_0(d,d,d,d)$, where $A_m(d)$ is the area of an isosceles hyperbolic triangle with two sides of length $d$, inscribed in a hyperbolic circle with its third side a diameter.
\end{lemma}

The notation in the statement and proof of this result comes from \cite{DeB_Voronoi}. In particular, for an $n$-tuple $(x_1,\hdots,x_n)\in(0,\infty)^n$ that records the side lengths of a convex \textit{cyclic} $n$-gon in $\mathbb{H}^2$, that is, one inscribed in a hyperbolic circle, $D_0(x_1,\hdots,x_n)$ records the area of this $n$-gon, compare \S 3.1 there. (It is a basic fact, proven in eg.~\cite{DeB_cyclic}, that a cyclic hyperbolic $n$-gon is determined up to isometry by its collection of side lengths.) 

	\begin{proof}
		By Lemma \ref{even vertices}, $C$ has an even number of vertices and therefore $k \geq 4$.  
		So Area$(C) \geq (k-2)A_m(D)$ by \cite[Theorem 3.31]{DeB_Voronoi}. 
		For an isosceles hyperbolic triangle $T$ with two sides of length $d$, inscribed in a hyperbolic circle with its third side a diameter, we observe that the union of $T$ with its reflection across this diameter is a hyperbolic square with side lengths $d$. Thus $2A_m(d) = D_0(d,d,d,d)$. It follows that if $k>4$ then 
\[ \mathrm{Area}(C) \geq (k-2)A_m(d) = (k-4)A_m(d) + 2A_m(d) = (k-4)A_m(d) + D_0(d,d,d,d), \]
and if $k=4$ then Area$(C) \geq 2A_m(d) = D_0(d,d,d,d)$.\end{proof}
	
The centered dual complex may not tile its entire underlying space, so we add a possibly empty collection of horocyclic ideal triangles to it to produce the \textit{centered dual plus}, see \S 5 of \cite{DeB_Voronoi}. This decomposition has underlying space $\mathbb{H}^2$; or after projecting to $F$, underlying space $F$. We will use it to prove our bound, which extends that of Proposition 0.2 of \cite{DeB_many} to the context of surfaces with boundary.

\begin{prop}\label{maximal rad}  Fix natural numbers $\chi<0$ and $b,n\ge 0$. For a complete, finite-area hyperbolic surface $F$ with compact geodesic boundary that has $n$ cusps, $b$ boundary components and Euler characteristic $\chi$, and any packing of $F$ by $k$ disks of radius $r>0$ ($k \in \mathbb{N}$), we have $r\leq r_{\chi, n, b}^k$, where $r_{\chi, n, b}^k$ is the unique solution to:

$$ 3\left(2-\left(\frac{2\chi+b+n}{k}\right) \right) \alpha(r_{\chi, n, b}^k) +\frac{2n}{k}\beta(r_{\chi, n, b}^k) + \frac{2b}{k} \gamma(r_{\chi, n, b}^k) = 2\pi. $$
Here, $\alpha(r) = 2\sin^{-1}{\left(\frac{1}{2\cosh{r}}\right)}, 
	\beta(r) = \sin^{-1}{\left(\frac{1}{\cosh{r}}\right)},$ and 
	$\gamma(r) = 2\sin^{-1}{\left(\frac{1}{\sqrt{2}\cosh{r}} \right)}$
	each measure vertex angles, respectively that of an equilateral triangle with side length $2r$, 
	at finite vertices of a horocyclic ideal triangle with compact side length $2r$, 
	and of a square with side length $2r$.

A complete, finite-area, $n$-cusped hyperbolic surface $F$ with $b$ geodesic boundary components has an equal-radius packing by disks of radius $r=r_{\chi, n, b}^k$ if and only if the disks are centered at the internal vertices of a decomposition of $F$ into equilateral triangles, $n$ horocyclic ideal triangles, all with compact side length $2r_{\chi, n, b}^k$, and $b$ Saccheri quadrilaterals with leg length $r_{\chi, n, b}^k$ and summit length $2r_{\chi, n, b}^k$, all with disjoint interiors, such that there are $k$ vertices in $\mathit{int}\,F$ and $\partial F$ is the union of the quadrilaterals' bases.
\end{prop}

The strategy of proof follows that of Proposition 0.2 of \cite{DeB_many}. The basic observation is that if $F$ has injectivity radius $r$ at $p$ then every edge of the centered dual decomposition plus of $F\cup\overline{F}$ determined by $\{p,\bar{p}\}$ has length at least $2r$. Lemma \ref{new bound} and related results then yield lower bounds on two-cell areas, which sum to more than the area of $F$ if $r$ exceeds $r_{\chi,n,b}^k$, or if equality holds and any such two-cells are not as simple as possible.  

\begin{proof}
	Let $F$ be a complete, finite-area hyperbolic surface of Euler characteristic $\chi$ with $n$ cusps and $b$ geodesic boundary components equipped with an equal-radius packing by $k$ disks of radius $r>0$.  Note that the set $\cals$ of disk centers is contained in $\mathit{int}\,F$ since $r$ is positive.
	For the double $\mathit{DF}$ we have $\chi(DF) = 2\chi(F)$. 
	Let $C_1, \ldots, C_m$ be the two-cells of the centered dual complex plus of $\cals\cup\overline\cals$, where $\overline{\cals} = r(\cals)$. By the Gauss-Bonnet theorem, we have
	\begin{equation}
	\mathrm{Area}(C_1) + ... + \mathrm{Area}(C_m) = -2\pi\chi(DF) = -4\pi\chi
	\end{equation}
	Order the cells so that $C_i$ is non-compact if and only if $i\leq m_0$ for some fixed $m_0 \leq m$, and for each $i$ let $n_i$ be the number of edges of $C_i$.
	
	For any non-compact cell ($i\leq m_0$,) we have by \cite[Theorem 4.16]{DeB_Voronoi}  that 
	$$\mathrm{Area}(C_i) \geq D_0(\infty, d, \infty) + (n_i-3)D_0(d,d,d),$$
	Notice that because  the radius-$r$ disks centered at the vertices of each cell are disjoint, each compact edge must have length at least $d:=2r$.  
	Equality holds if and only if $n_i=3$, i.e.~$C_i$ is a horocyclic ideal triangle, and the compact side length is $d$.
	
	For each compact cell ($m_0 < i \leq m$,) we have two cases.
	If $C_i$ does not intersect the axis of reflection of $\mathit{DF}$, then we apply  \cite[Theorem 3.31, Corollary 3.5]{DeB_Voronoi} to obtain the bound
	$$\mathrm{Area}(C_i) \geq (n_i-2)A_m(d) \geq (n_i-2)D_0(d,d,d).$$ Here equality holds if and only if $C_i$ is an equilateral triangle with side length $d$.
	
	On the other hand, if $C_i$ does intersect the axis of reflection, then by Lemma \ref{new bound} and \cite[Corollary 3.5]{DeB_Voronoi} we have
	$$\mathrm{Area}(C_i) \geq (n_i-4)A_m(d) + D_0(d,d,d,d) \geq (n_i-4)D_0(d,d,d) + D_0(d,d,d,d),$$
	with equality if and only if $n_i = 4$, i.e.~$C_i$ is a square with side length $d$.  Note that this is strictly greater than $(n_i-2) D_0(d,d,d)$, since $D_0(d,d,d,d) > 2 D_0(d,d,d)$.
	
	We claim there are at least $b$ such cells. By Lemma \ref{boundary is Voronoi} there is at least one Voronoi vertex in each component of $\partial F$. For such a vertex $v$, let $C_v$ be the centered dual two-cell containing $v$, either a centered Delaunay cell geometrically dual to $v$ or a union of Delaunay two-cells dual to vertices of a component $T$ of the non-centered Voronoi subgraph. The claim will follow from the fact that in the latter case, $T$ contains no Voronoi vertex in any other component of $\partial F$, since this implies that $C_v \neq C_w$ if $v$ and $w$ lie in different components of $\partial F$. By Lemma \ref{symmetric trees}, $T$ is preserved by the reflection $r$ across $\mathit{DF}$, so if it contained Voronoi vertices $v$ and $w$ in different components of $\partial F$ then for any edge path $\gamma$ in $T$ joining $v$ to $w$, $\gamma\cup r(\gamma)$ would form a cycle in $T$, contradicting that it is a tree. This proves the fact and hence the claim.
	
	Applying the bounds above to our Gauss-Bonnet formula results in
	\begin{align*}
	 -4\pi\chi  & \geq \sum_{i=1}^{m_0} \left( D_0(\infty, d, \infty) + (n_i-3)D_0(d,d,d) \right) \\
			&	\qquad		+ \sum_{i=m_0+1}^{m_0+b} \left((n_i-4)D_0(d,d,d) + D_0(d,d,d,d)\right)
	 							+ \sum_{i=m_0+b+1}^m \left((n_i-2)D_0(d,d,d)  \right) \end{align*}
By the comments above, equality holds here only if $n_i = 3$ for each $i\leq m_0$, $n_i =4$ for $m_0< i\leq m_0+b$, and there are exactly $b$ cells intersecting the axis of reflection.  Rearranging the inequality now yields:
	\begin{align*}
	-4\pi\chi	& \geq m_0 \cdot D_0(\infty, d, \infty) + \left( \sum_{i=1}^m (n_i-2) - m_0 - 2b\right) \cdot D_0(d,d,d) + b\cdot D_0(d,d,d,d)\\
			& \geq 2n \cdot D_0(\infty, d, \infty)  + \left( \sum_{i=1}^m n_i-2m-2b -2n\right) \cdot D_0(d,d,d) + b\cdot D_0(d,d,d,d)\\
			& = 2n \cdot (\pi - 2\beta(r)) +(2e-2m-2b-2n)\cdot (\pi - 3\alpha(r)) +b \cdot (2\pi - 4\gamma(r))\\
			& = (2e-2m) \pi -4n\beta(r) - 3(2e-2m-2b-2n)\alpha(r) - 4b\gamma(r) \\
			& = (4k-4\chi) \pi -4n\beta(r) - 3(4k-4\chi-2b-2n)\alpha(r) - 4b\gamma(r)
	 \end{align*}

	 To move from the first to second line we note that $m_0\geq 2n$, since each non-compact cell has only one ideal vertex and $\mathit{DF}$ has $2n$ cusps, and rewriting $m_0$ as $2n +m_0 -2n$, we use the fact that $D_0(\infty, d, \infty) > D_0(d,d,d)$. Thus equality holds between these lines if and only if $m_0 = 2n$. To move to the last line we compactify $\mathit{DF}$ by filling in each cusp with a unique point then perform an Euler characteristic computation: 
	 \[ 2\chi+2n = v-e+f = (2k+2n)-e+m \implies 2e-2m = 4k-4\chi. \]	
	 After further rearrangement, we have that
	 \begin{align*}
	 	 4k\pi &\leq 4n\beta(r)+3(4k-4\chi-2b-2n)\alpha(r) +4b\gamma(r)\\
			\implies 2\pi & \leq 3\left(2-\left(\frac{2\chi+b+n}{k}\right) \right) \alpha(r) +\frac{2n}{k}\beta(r) + \frac{2b}{k} \gamma(r)
	 \end{align*}
	 
	 Since $\alpha, \beta,$ and $\gamma$ are decreasing functions in $r$ we see that $r \leq r_{\chi,n,b}^k$ 
	 and equality holds if and only if the compact cells away from the boundary are triangles with compact side length $d= 2r_{\chi,n,b}^k$, the cells intersecting the boundary are squares with side length $d$, there are exactly $b$ of these, and there are exactly $2n$ non-compact cells, each a horocyclic ideal triangle with compact side length $d$. Since the entire complex is $r$-invariant, in the case that equality holds the square cells are divided into symmetric Saccheri quadrilaterals by their intersections with $\partial F$, each with summit length $2r_{\chi, n, b}^k$ and leg length $r_{\chi, n, b}^k$. And exactly $n$ of the horocyclic ideal triangles (and half of the equilateral triangles) lie in $F\subset\mathit{DF}$.
	 
	 At this point, we have showed that the radius $r$ of a $k$-disk packing satisfies $r\leq r_{\chi, n, b}^k$ and equality holds if and only if the Delaunay tessellation of $\mathit{DF}$ yields the prescribed decomposition. It remains to show that beginning with such a decomposition, there exists a packing of $F$ by $k$ disks of radius $r_{\chi, n, b}^k$, i.e. that the disks are embedded without overlapping.

	 Let us define a \textit{sector} of a metric disk as its intersection with two half-planes whose boundaries contain the disk's center. We will center the disks at the internal vertices of the given decomposition. To guarantee they form a packing, we will show for each polygon $P$ that if $B$ is a disk centered at a vertex $p$ of $P$, then $B\cap P$ is a sector of $B$ and $B\cap P$ does not overlap $B'\cap P$ for any disk $B'$ centered at another vertex $p'$ of $P$.
	
	Each equilateral triangle $T$ is a centered polygon with side lengths $2r_{\chi,n,b}^k$, thus \cite[Lemma 5.12]{DeB_Voronoi} guarantees that the open disks of radius $r$ centered at each vertex of the triangle intersect with $T$ in full sectors and that these disks do not overlap in $T$ for any $r\leq r_{\chi, n, b}^k$. 
	
	For a horocyclic ideal triangle $R$ with compact side length $2r_{\chi,n,b}^k$, notice that the perpendicular bisector $\gamma$ of the compact side $\rho$ is an axis of reflection for $R$. It follows from the hyperbolic law of cosines that the closest point of $\gamma$ to either vertex of $R$ on $\rho$ is the point at $\gamma \cap \rho$. Since the length of $\rho$ is $2r_{\chi,n,b}^k$, we are guaranteed that $\gamma$ divides a disk of radius $r\leq r_{\chi, n, b}^k$ centered at one vertex of the side $\rho$ of $R$ from the corresponding disk centered at the other such vertex, so they do not overlap. And for either such disk $B$, the full sector of $B$ determined by the geodesics containing $\rho$ and the non-compact side of $R$ containing its center $v$ lies in $R$, since $\gamma$ divides this sector from the side of $R$ opposite $v$.
	
The result follows from the fact that for any triangle $T$ and a disk $B$ of radius $r$ centered at a vertex $v$ of $T$, $B\cap T$ is a full sector of $B$ if and only if $d(v,\gamma) \leq r$, where $\gamma$ is the side of $T$ opposite $v$.
	
	
	For a Saccheri quadrilateral $Q$, again we see that the perpendicular bisector $\gamma$ of the summit $\sigma$ is an axis of reflection for $Q$, as its legs have equal length $r_{\chi, n, b}^k$. Thus the closest point of $\gamma$ to either vertex of $Q$ on $\sigma$ is $\gamma \cap \sigma$. Since the length of $\sigma$ is $2r_{\chi, n, b}^k$, we are guaranteed that $\gamma$ divides a disk of radius $r\leq r_{\chi, n, b}^k$ centered at one summit vertex of $Q$ from the corresponding disk centered at the other such vertex. Additionally, as the base of $Q$ is perpendicular to the legs of $Q$ by the definition of a Saccheri quadrilateral, the vertices at the base are also the closest points from the summit vertices on the base. Thus for any disk $B$ of radius $r\leq r_{\chi, n, b}^k$ centered at a vertex on the summit, $B\cap Q$ is a full sector.
	
Now suppose we are given a complete, finite-area, $n$-cusped hyperbolic surface $F$ with $b$ geodesic boundary components decomposed into equilateral triangles, $n$ horocyclic ideal triangles, all with compact side length $2r_{\chi, n, b}^k$, and $b$ Saccheri quadrilaterals with leg length $r_{\chi, n, b}^k$ and summit length $2r_{\chi, n, b}^k$, all with disjoint interiors, such that there are $k$ vertices in $\mathit{int}\,F$ and $\partial F$ is the union of the quadrilaterals' bases. The observations above show that $F$ has a packing by disks of radius $r_{\chi,n,b}^k$ centered at the internal vertices of this decomposition, since we can assemble each disk from its sectors of intersection with the polygons of the decomposition and show that the resulting collection is embedded without overlapping by inspecting the polygons.
\end{proof}
	
\begin{thm}\label{maximal injrad}
Fix natural numbers $\chi<0$ and $b,n\ge 0$. For a complete, finite-area hyperbolic surface $F$ with compact geodesic boundary that has $n$ cusps, $b$ boundary components and Euler characteristic $\chi$, and any $p\in F$, $\mathit{injrad}_p(F) \leq r_{\chi, n, b}$, where $r_{\chi, n, b}$ is the unique solution to:
$$ 3\left(2-\left(2\chi+b+n\right) \right) \alpha(r_{\chi, n, b}) +2n\beta(r_{\chi, n, b}) + 2b \gamma(r_{\chi, n, b}) = 2\pi. $$
Here, $\alpha(r) = 2\sin^{-1}{\left(\frac{1}{2\cosh{r}}\right)}$,  
	$\beta(r) = \sin^{-1}{\left(\frac{1}{\cosh{r}}\right)}$, and 
	$\gamma(r) = 2\sin^{-1}{\left(\frac{1}{\sqrt{2}\cosh{r}} \right)}$
	each measure vertex angles, respectively that of an equilateral triangle with side length $2r$, 
	at finite vertices of a horocyclic ideal triangle with compact side length $2r$, 
	and of a square with side length $2r$.

For a complete, finite-area, $n$-cusped hyperbolic surface $F$ with $b$ geodesic boundary components and Euler characteristic $\chi$, and $p\in F$, $\mathit{injrad}_p(F) = r_{\chi,n,b}$ if and only if $p$ is the unique internal vertex of a decomposition of $F$ into equilateral triangles with side length $2r_{\chi, n, b}$, $n$ horocyclic ideal triangles with compact side length $2r_{\chi, n, b}$, and $b$ Saccheri quadrilaterals with leg length $r_{\chi, n, b}$ and summit length $2r_{\chi, n, b}$, all with disjointly embedded interiors, such that $\partial F$ is the union of the quadrilaterals' bases. There does exist such a pair $(F,p)$ with $\mathit{injrad}_p(F) = r_{\chi,n,b}$.
\end{thm}

\begin{proof} The $k=1$ case of Proposition \ref{maximal rad} gives the Theorem's upper bound and characterizes the surfaces that attain it. To prove the Theorem in full, it remains to give examples of such maximizing surfaces. We do this below.

Fix $\chi < 0, \ n \geq 1, \ b \geq 1$ such that $g=\frac12 (2-\chi-n-b)$ is a non-negative integer. Let $r_{\chi, n, b}$ be as defined in Corollary \ref{maximal injrad}.
Take $4g+n+b-2$ equilateral triangles each with side length $2r_{\chi, n, b}$ and arrange them fan-like with common vertex $v$ to form a triangulated $(4g+n+b)$-gon, $P_0$ in $\mathbb{H}^2$.
Label the edges cyclically $a_1,\ b_1,\ c_1,\ d_1, \ldots, a_g,\ b_g,\ c_g,\ d_g, e_1, \ldots, e_n, \ f_1, \ldots , f_b$ and give each edge the counter-clockwise orientation. See Figure \ref{fig:polygon1}.

\begin{figure}
    \centering
    \begin{subfigure}[b]{0.3\textwidth}
        \includegraphics[scale=1.2]{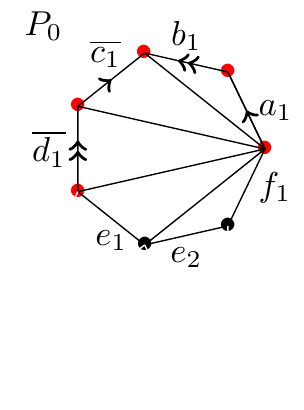}
        \caption{Step 1: genus}
        \label{fig:polygon1}
    \end{subfigure}
    ~ 
    \begin{subfigure}[b]{0.3\textwidth}
        \includegraphics[scale=1.2]{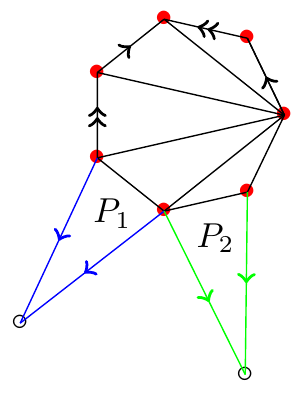}
        \caption{Step 2: cusps}
        \label{fig:polygon2}
    \end{subfigure}
       \begin{subfigure}[b]{0.3\textwidth}
        \includegraphics[scale=1.2]{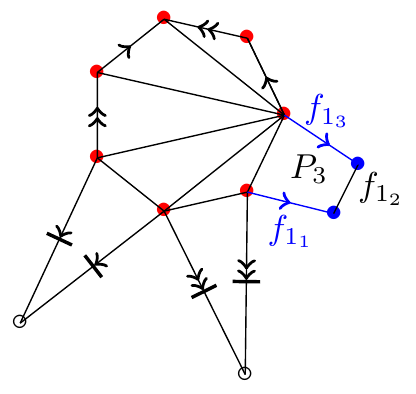}
        \caption{Step 3: boundary}
        \label{fig:polygon3}
    \end{subfigure}
    \caption{Construction of a surface with $\chi = -3, \ n=2, \ b=1$}
    \label{polygons}
\end{figure}

For each $1 \leq i \leq g$, let $A_i$ be the orientation-preserving isometry paring the edge $a_i$ with $\overline{c_i}$ (the segment $c_i$ with clock-wise orientation) and similarly define $B_i$ as the orientation-preserving isometry paring $b_i$ with $\overline{d_i}$. See Figure \ref{fig:polygon2}.
Notice that the result has made the first $4g+1$ vertices equivalent to $v$ under the quotient. 

Next, for $1\leq j \leq n$ we attach to each $e_j$ a horocyclic ideal triangle $P_j$ with base length $2r_{\chi, n, b}$ and let $E_j$ be the parabolic isometry fixing the respective ideal vertex and taking one of the infinite sides of the triangle to the other. See Figure \ref{fig:polygon3}.
Thus we have added the next $2n$ vertices as equivalent to $v$ under the quotient.

Finally for $1 \leq k \leq b$ we attach to each $f_k$ a Saccheri quadrilateral $P_{n+k}$ with summit length $2r_{\chi, n, b}$ and leg length $r_{\chi, n, b}$, assigning counter-clockwise orientation to the additional three sides, $f_{k_1}, f_{k_2}, $ and $f_{k_3}$.  Denote by $F_k$ the orientation-preserving isometry which exchanges opposite sides $f_{k_1}$ and $\overline{f_{k_3}}$, thus making the final $2b-2$ vertices of $P_0$ also equivalent to $v$ under quotient.

By the definition of $r$ taken from the Corollary \ref{maximal injrad}, the sum of the angles about vertex $v$ is
	$$ (4g+n+b-2) \cdot 3\alpha(r_{\chi, n, b}) + n \cdot 2\beta(r_{\chi, n, b}) + b \cdot 2\gamma(r_{\chi, n, b}) = 2\pi.$$

Thus Poincare's Polygon Theorem tells us $G = \left< A_1,\ B_1, \ldots, A_g, \ B_g, \ E_1, \ldots, E_n, F_1, \ldots, F_b \right>$ is a Fuchsian group with the polygon $P:= \bigcup_{i=0}^{n+b} P_i$ its fundamental domain.  Hence the quotient $F = \mathbb{H}^2 / G$ is a complete hyperbolic surface. 
Inspecting the edge pairing, we see that $F$ has $n$ cusps and $b$ boundary components. Thus by the Gauss-Bonnet theorem it also has genus $g$, as desired. By Proposition \ref{maximal rad}, $F$ has injectivity radius $r_{\chi,n,b}$ at the quotient of $v$.\end{proof}

\begin{remark}\label{boundary length} By the trigonometry of Saccheri quadrilaterals (cf.~Lemma \ref{loop length}), each boundary component of the maximizing surface $F$ has length $2\sinh^{-1}{\left(\tanh{(r_{\chi,n,b})}\right)}$.\end{remark}

\section{The systole of loops}\label{sysloops}

Now we turn our attention to the systole of loops function.  Recall from the introduction that this is defined at a point $p$ of a hyperbolic surface $F$ (possibly with geodesic boundary) as the infimum $\mathit{sys}_p(F)$ of the lengths of closed, non-constant geodesic arcs in $F$ based at $p$.  The main result of this section implies that unlike $\mathit{injrad}_p(F)$, $\mathit{sys}_p(F)$ is not bounded above over all bounded surfaces $F$ with a fixed topology and all $p\in F$:

\begin{prop}\label{up up up}  For a complete, connected hyperbolic surface $F$ with compact geodesic boundary and Euler characteristic $\chi$, and a component $\gamma$ of $\partial F$ with length $\ell$, there is a point $p\in\gamma$ such that $\mathit{sys}_p(F) \geq \frac{\ell}{-6\chi}$.\end{prop}

\noindent The proof is simple-minded and direct. We will give it first, then justify the assertions therein.

\begin{proof} For $F$ and $\gamma$ as above, define the \textit{cut locus} $\Sigma$ of $\gamma$ as the set of points of $F$ that have more than one shortest arc to $\gamma$. By Proposition \ref{quad decomp} below, $\Sigma$ determines a decomposition of $F$ as a finite union of Saccheri quadrilaterals. Each of these has its base in $\gamma$ and its summit in $\Sigma\cup(\partial F - \gamma)$.  If $F$ has cusps then some vertices of these quadrilaterals are ideal, as each cusp of $F$ has some edges of $\Sigma$ running out of it.  For ease of bookkeeping we will compactify $F$ by adding a point at the end of each cusp, producing a surface $\bar{F}$ with boundary that has Euler characteristic $\chi+n$ (if $F$ has $n$ cusps) and a decomposition into compact quadrilaterals.

We first observe that the number of these quadrilaterals is universally bounded above by $-6\chi$. This will follow from the Euler characteristic formula $\chi + n = v-e+f$.  We divvy up the number $v$ of vertices and $e$ of edges as follows:\begin{align*}
	& v = v_b+v_c+ n && e = e_b+e_p + e_c. \end{align*}
Here $v_b$ and $e_b$ are the number of \textit{boundary} vertices and edges, respectively --- those contained in $\partial F$; $v_c$ is the number of vertices in the interior of $F$, each belonging to the cut locus, and $e_c$ is the number of cut locus edges; $n$ is the number of additional vertices at the ends of cusps; and $e_p$ is the number of edges that have one endpoint in $\gamma$, the legs of the Saccheri quadrilaterals. Since $\partial F$ is compact it is a disjoint union of circles, so $v_b = e_b$. Since each Saccheri quadrilateral has two legs, and each leg is in two quadrilaterals, we have $e_p = f$.  Plugging this into the Euler characteristic formula and simplifying yields $\chi = v_c -e_c$.

Each interior cut locus vertex is contained in at least three edges, whereas each cut locus edge has at most two endpoints at such vertices, so $3v_c\leq 2e_c$, yielding $\chi \leq-e_c/3$.  But each cut locus edge is contained in two Saccheri quadrilaterals, whereas each quadrilateral has at most one such edge in its boundary, giving $e_c\leq f/2$ and hence $\chi\leq -f/6$ as claimed.

The second observation is that any homotopically nontrivial loop based in a quadrilateral $Q$ must exit and re-enter that quadrilateral through its legs, if its summit lies in $\partial F - \gamma$; or if the summit lies in $\Sigma$ then it must exit $Q\cup\overline{Q}$, for $\overline{Q}$ as in Lemma \ref{quads}, through the union of their legs.  This is because these regions are simply connected and their frontiers in $F$ are unions of legs (since each quadrilateral has its base in $\gamma$).

If $\ell$ is the length of $\gamma$ then the first observation implies that there exists a quadrilateral $Q$ with base length at least $\frac{\ell}{-6\chi}$, since each quadrilateral has its base (and no other edges) in $\gamma$.  A basepoint $x$ located at the midpoint of the base of this quadrilateral has distance at least $\frac{\ell}{-12\chi}$ to the frontier of $Q$ or $Q\cup\overline{Q}$, in the respective cases above. If $Q$ has its summit in $\partial F$ this is clear, since the nearest-point projection from $Q$ to its base is well known to be distance-reducing.  In the other case we note that any arc in $Q\cup\overline{Q}$ can be replaced by one in $Q$ with the same length by replacing each interval of its intersection with the interior of $\overline{Q}$ by the image of this interval under the reflective involution that exchanges $\overline{Q}$ and $Q$. This implies in particular that the distance from $x$ to either leg of $\overline{Q}$ is at least $\frac{\ell}{-12\chi}$.  The second observation now implies the result.\end{proof}

In the rest of this section we will prove the assertions about the cut locus used in the proof of Proposition \ref{up up up}.  The cut locus is ``la ligne de partage'' in Bavard's paper \cite[\S 1]{Bavard_anneaux} pointed out to us by the referee. Our account here largely parallels those in the work of Bowditch--Epstein \cite{BowEp} and Kojima \cite{Kojima}, in different but closely related contexts. To begin, cut loci are locally finite cell complexes that have codimension one in their ambient manifolds and cells that are polygons.

\begin{lemma}\label{cut locus} For a complete, connected, finite-area hyperbolic surface $F$ with geodesic boundary and a compact component $\gamma$ of $\partial F$, define the \mbox{\rm cut locus} $\Sigma$ of $\gamma$ to be the set of points with more than one shortest arc to $\gamma$.  Taking the vertex set $\mathcal{V}$ of $\Sigma$ to be the union of $\Sigma\cap(\partial F - \gamma)$ (the \mbox{\rm boundary vertices}) with the set of points of $\mathit{int}\,F$ that have at least three shortest arcs to $\gamma$ (\mbox{\rm interior vertices}), and an \mbox{\rm edge} to be the closure of a connected component of $\Sigma-\mathcal{V}$, $\Sigma$ is a locally finite graph with geodesic edges. Each interior vertex has valence at least three.\end{lemma}

\begin{proof}  Let $p\co\widetilde{F}\to F$ be the universal cover. We may take $\widetilde{F}$ to be a convex subset of $\mathbb{H}^2$ bounded by a locally finite disjoint union of geodesics. For any two of these boundary geodesics $\tilde\gamma_1$ and $\tilde\gamma_2$, each projecting to $\gamma$, the locus of points equidistant from $\tilde\gamma_1$ and $\tilde\gamma_2$ is itself a geodesic. This is in fact a well known property of any two disjoint geodesics in $\mathbb{H}^2$.  If $\tilde\gamma_1$ and $\tilde\gamma_2$ do not share an endpoint at infinity then their equidistant is the perpendicular bisector of the unique shortest geodesic arc joining them; if they do share an ideal endpoint $v$ then it bisects every arc of a horocycle based at $v$ that joins $\tilde\gamma_1$ to $\tilde\gamma_2$.

We claim that the preimage $\widetilde{\Sigma}$ of $\Sigma$ lies in the union of such equidistant geodesics; that any point contained in a unique one has a neighborhood in $\widetilde{\Sigma}$ consisting of points with the same property; and that the set of points contained in more than one is discrete.

For any $x\in\Sigma$ and any $\tilde{x}\in p^{-1}(x)$, each shortest geodesic arc from $x$ to $\gamma$ lifts to a shortest geodesic arc in $\widetilde{F}$ from $\tilde{x}$ to a component $\tilde\gamma$ of $\partial\widetilde{F}$ projecting to $\gamma$.  The lifts of distinct such arcs terminate at distinct components of $\partial\widetilde{F}$, since there is only one shortest arc from $\tilde{x}$ to any geodesic that does not contain it.  Therefore since $x$ has at least two distinct shortest arcs to $\gamma$, $\tilde{x}$ is contained in the equidistant locus $\lambda_{12}$ to at least two components $\tilde\gamma_1$ and $\tilde\gamma_2$ of $p^{-1}(\gamma)$.

If $x$ has only two shortest arcs to $\gamma$, and these have length $\ell$, then since the collection of boundary geodesics of $\widetilde{F}$ is locally finite there exists $\epsilon>0$ such that every other component $\tilde\gamma$ of $p^{-1}(\gamma)$ has distance at least $\ell+\epsilon$ from $\tilde{x}$.  Thus every point in the $\epsilon$-neighborhood $U_{\epsilon}$ of $\tilde{x}$ is closer to $\gamma_1$ and/or $\gamma_2$ than to any other lift of $\gamma$, and $U_{\epsilon}\cap\widetilde{\Sigma} = U_{\epsilon}\cap\lambda_{12}$.

If on the other hand $\tilde{x}$ is closest to lifts $\tilde\gamma_1,\tilde\gamma_2,\hdots,\tilde\gamma_k$ for some $k>2$, that is, if $x$ has more than two shortest arcs to $\gamma$, then $\tilde{x}$ lies in the intersection of the equidistants $\lambda_{ij}$ for $1\leq i < j\leq k$.  But any two of these intersect only at $\tilde{x}$, since \textit{any} two geodesics in $\mathbb{H}^2$ intersect at most once, and as above by local finiteness there exists $\epsilon>0$ such that 
\[ U_{\epsilon}\cap\widetilde{\Sigma} \subset U_{\epsilon}\cap \bigcup\{\lambda_{ij}\,|\,1\leq i<j\leq k\}. \]
This establishes the claim.  Taking the edges of $\Sigma$ to be projections of segments of the form $\lambda_{ij}\cap\widetilde{\Sigma}$ we also find that $\Sigma$ is a locally finite graph, since for every point $x$ the neighborhood $U_{\epsilon}$ of $\tilde{x}$ intersects at most one vertex and at most finitely many edges.  To see that each vertex $x$ has valence at least three, we  observe that if the lifts $\tilde\gamma_1,\hdots,\tilde\gamma_k$ of $\gamma$ nearest to $\tilde{x}$ are enumerated in cyclic order then $U_{\epsilon}\cap\Sigma$ contains a segment of $\lambda_{i\,i+1}$ for each $i<k$, and of $\lambda_{k1}$.\end{proof}

\begin{lemma}\label{quads} For a complete, connected hyperbolic surface $F$ with compact geodesic boundary, a component $\gamma$ of $\partial F$, and an edge $e$ of the cut locus $\Sigma$ of $\gamma$, the set
\[ \bigcup_{x\in\mathit{int}\,e} \{\lambda\,|\,\lambda\ \mbox{is a shortest arc from $x$ to}\ \gamma\} \]
is the embedded image of $\mathit{int}(Q\cup\overline{Q})$ under the universal cover $\pi\co\widetilde{F}\to F$, where $Q$ and $\overline{Q}$ are Saccheri quadrilaterals in $\widetilde{F}$ with their bases in $p^{-1}(\gamma)\subset\partial\widetilde{F}$ and with $Q\cap\overline{Q}=\tilde{e}$ the summit of each, for some lift $\tilde{e}$ of $e$ to $\widetilde{F}$.  

The hyperbolic reflection in the geodesic containing $Q\cap\overline{Q}$ exchanges $Q$ with $\overline{Q}$. Furthermore, if $Q'\cap\overline{Q}'$ corresponds to a distinct edge $f$ of $\Sigma$ then 
\[ \pi(\mathit{int}(Q\cup\overline{Q}))\cap \pi(\mathit{int}(Q'\cup\overline{Q}')) = \emptyset. \]
\end{lemma}

\begin{proof} We regard $\widetilde{F}$ as a convex subset of $\mathbb{H}^2$ bounded by the disjoint union of geodesics $p^{-1}(\partial F)$. Let $\tilde{e}$ be a lift of $e$ to an edge of $\widetilde{\Sigma} = p^{-1}(\Sigma)$.  For any $x\in\mathit{int}\,e$, by Lemma \ref{cut locus} there are exactly two shortest arcs $\lambda_1$ and $\lambda_2$ joining $x$ to $\gamma$.  If $\tilde\lambda_1$ and $\tilde\lambda_2$ are their respective lifts to $\widetilde{F}$ based at $\tilde{x}\in\tilde{e}$ then they terminate at different components $\tilde\gamma_1$ and $\tilde\gamma_2$ of $p^{-1}(\gamma)$.  As argued in the proof of Lemma \ref{cut locus}, all of $\tilde{e}$ is contained in the geodesic consisting of points equidistant from $\tilde\gamma_1$ and $\tilde\gamma_2$.  And for any $\tilde{x}\in\tilde{e}$, the reflection in this geodesic exchanges the shortest arcs from $\tilde{x}$ to $\tilde\gamma_1$ and $\tilde\gamma_2$; that is, it exchanges $\tilde\lambda_1$ and $\tilde\lambda_2$.

We note that the geodesic containing $\tilde{e}$ does not share an endpoint at infinity with $\tilde\gamma_1$ or $\tilde\gamma_2$ since these geodesics do not share endpoints, $\gamma$ being compact.

Let $r_1\co\mathbb{H}^2\to\tilde\gamma_1$ be the nearest-point retraction to the geodesic $\tilde\gamma_1$.  For any $\tilde{x}\in\mathbb{H}^2-\tilde\gamma_1$, if $\tilde\lambda$ is the shortest arc from $\tilde{x}$ to $\tilde\gamma_1$ then $r_1$ takes all of $\tilde\lambda$ to its endpoint in $\tilde\gamma_1$, and in fact $r_1^{-1}(r_1(\tilde{x}))$ is the geodesic containing $\tilde\lambda$.  By standard hyperbolic trigonometric arguments, $r_1$ takes the geodesic containing $\tilde{e}$ injectively to an interval with compact closure, since as noted above this geodesic does not share an endpoint with $\tilde{\gamma}_1$. The image of $\mathit{int}\,\tilde{e}$ under $r_1$ is thus an open sub-interval with compact closure $I$.

We let $Q$ be the region in $r_1^{-1}(I)$ between $\tilde\gamma_1$ and the geodesic containing $\tilde{e}$. Each point in the interior of $I$ is the image of a unique point $\tilde{x}$ in $\tilde{e}$, and by the above its contribution to $Q$ is the shortest arc $\tilde\lambda_1$ joining $\tilde{x}$ to $\tilde\gamma_1$.  The frontier of $Q$ in $\widetilde{F}$ is the union $\tilde{e}$ with $r_1^{-1}(\partial I)$, which is itself a disjoint union of two geodesic arcs that each meet $\tilde\gamma_1$ perpendicularly.  Thus $Q$ is a Saccheri quadrilateral with base $I$, legs the components of $r_1^{-1}(I)$, and summit $\tilde{e}$.  Reflection in the geodesic containing $\tilde{e}$ further takes $Q$ to the corresponding object $\overline{Q}$ for $\tilde\gamma_2$, since it exchanges the arcs $\tilde\lambda_1$ and $\tilde\lambda_2$ starting at $\tilde{x}$ for each $\tilde{x}\in\mathit{int}\,e$.

It remains to show that the interior of $Q\cup\overline{Q}$ is embedded by $\pi$, and that it does not intersect the image of $\mathit{int}(Q'\cup\overline{Q}')$ for $Q$ and $Q'$ corresponding to any edge $f$ of $\Sigma$ distinct from $e$.  These assertions both follow from the fact that for any edge $\tilde{f}$ of $\widetilde{\Sigma}$ distinct from $\tilde{e}$, the union of quadrilaterals $Q'\cup\overline{Q}'$ that intersect at $\tilde{f}$ does not overlap $Q\cup\overline{Q}$: for the first assertion, take $\tilde{f}$ to be a lift of $e$ distinct from $\tilde{e}$, and for the second let it lift $f\neq e$.  The non-overlapping of these quadrilateral unions follows in turn from the standard ``surgery'' argument below. 

Suppose that $\lambda$ and $\lambda'$ are shortest arcs from $x\in\tilde{e}$ and $y\in\tilde{f}$ to $p^{-1}(\gamma)$, respectively, that cross at some point $z$.  Taking $\lambda_0$ and $\lambda'_0$ to be their respective sub-arcs from $x$ and $y$ to $z$, and $\lambda_1$ and $\lambda_1'$ the complementary arcs, either $\lambda_0\cup\lambda_1'$ is at most as long as $\lambda$ or $\lambda_0'\cup\lambda_1$ is at most as long as $\lambda'$.  But since $\lambda$ intersects $\lambda'$ transversely, each of $\lambda_0\cup\lambda_1'$ and $\lambda_0'\cup\lambda_1$ has a strictly shorter geodesic arc to $p^{-1}(\gamma)$ in its proper homotopy class, contradicting the hypothesis that both $\lambda$ and $\lambda'$ were shortest.
\end{proof}

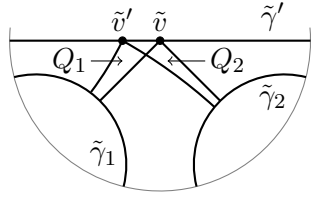
\begin{figure}
\begin{tikzpicture}

\draw [thin, gray] (-1.9924,0.1743) arc (175:365:2);
\draw [thick] (-2,0) -- (2,0);
\draw [fill] (0,0) circle [radius=0.05];
\node [above] at (0,0) {$\tilde{v}$};
\node [above] at (1.5,0) {$\tilde\gamma'$};

\draw [thick] (1.94,-0.485) arc (75.96:194.04:1.2);
\node at (1.5,-0.75) {$\tilde\gamma_2$};
\draw [thick] (-1.94,-0.485) arc (104.04:-14.04:1.2);
\node at (-0.75,-1.5) {$\tilde\gamma_1$};
\draw [thick] (0,0) -- (0.8,-0.8);
\draw [thick] (0,0) -- (-0.8,-0.8);

\draw [fill] (-0.5,0) circle [radius=0.05];
\node [above] at (-0.5,0) {$\tilde{v}'$};
\draw [thick] (-0.5,0) arc (-25.95:-37:4.164);
\draw [thick] (-0.5,0) arc (60.67:47.7:6.675);


\node at (-1.2,-0.27) {$Q_1$};
\draw [->] (-0.92,-0.27) -- (-0.5,-0.27);
\node at (0.9,-0.27) {$Q_2$};
\draw [<-] (0.1,-0.27) -- (0.6,-0.27);

\end{tikzpicture}

\caption{Quadrilaterals determined by points of $\Sigma\cap\partial F$, for $k=2$.}
\label{two quads}
\end{figure}

\begin{lemma} For $F$, $\gamma$, and $\Sigma$ as in Lemma \ref{cut locus}, $\Sigma$ is properly embedded in $F$ and intersects each compact component of $\partial F-\gamma$. A sufficiently small neighborhood of each rank-one cusp $c$ of $F$ intersects $\Sigma$ in a non-empty disjoint union of geodesic rays that exit $c$.\end{lemma}

\begin{proof} For a component $\gamma'$ of $\partial F-\gamma$, the fact that points of $\Sigma\cap \gamma'$ are isolated is a consequence of the hyperbolic trigonometry of quadrilaterals with two adjacent right angles: for a convex hyperbolic quadrilateral $Q$ with two right angles along a side in a geodesic $\tilde\gamma$, and a vertex $\tilde{v}$ of $Q$ joined to $\tilde\gamma$ by a side of length $b>0$ from $\tilde\gamma$, let $\alpha$ be the interior angle of $Q$ at $\tilde{v}$.  If the other vertex $\tilde{v}'$ of $Q$ off $\tilde\gamma$ is joined to it by a side of length $a>0$, and the side of $Q$ opposite $\tilde\gamma$ has length $c$ then by a version of the hyperbolic law of cosines proven in \cite[Ch.~VI.3.3]{Fenchel} we have:\begin{align}\label{quad}
	\sinh a = \sinh b\cosh c - \cosh b\sinh c \cos\alpha. \end{align}
This implies in particular, for fixed $b$ and $c$, that $a$ increases with $\alpha$.  

Given a point $v$ of $\Sigma\cap\gamma'$ that has $k$ shortest arcs to $\gamma$, each of length $b>0$, a lift $\tilde{v}$ of $v$ to the universal cover has $k$ different lifts $\tilde\gamma_1,\hdots,\tilde\gamma_k$ of $\gamma$ at distance $b$ from it.  For a point $\tilde{v}'$ on the lift $\tilde\gamma'$ of $\gamma'$ containing $v$, and each $i\in\{1,\hdots,k\}$, there is a convex hyperbolic quadrilateral $Q_i$ with right angles along a side in $\tilde\gamma_i$ and its opposite side the arc from $\tilde{v}$ to $\tilde{v'}$ in $\tilde\gamma'$. See Figure \ref{two quads}. Enumerating these as $Q_1,\hdots,Q_k$, $k\geq 2$, so that the interior angle $\alpha_1$ of $Q_1$ at $\tilde{v}$ is less than the corresponding angle $\alpha_2$ of $Q_2$ and so on, (\ref{quad}) implies that the sequence $a_1,\hdots,a_k$ is increasing, where for each $i$, $a_i$ is the length of the side of $Q_i$ joining $\tilde{v}'$ to $\tilde\gamma'$. But if $\tilde{v}'$ is near enough to $v$ then the shortest arc from its projection $v'$ to $\gamma$ is the projection of one of these sides; hence $v'$ has a unique shortest arc to $\gamma$.

We now show that each compact component $\gamma'$ of $\partial F-\gamma$ does intersect $\Sigma$.  To this end let us fix $v\in\gamma'$, which we may assume has a unique shortest arc $\lambda$ to $\gamma$, of length $b$. Then as in the proof of Lemma \ref{cut locus}, the next shortest arc from $v$ to $\gamma$ has length at least $b+\epsilon$ for some $\epsilon>0$.  If $v'\in\gamma'$ lies within distance $\epsilon/2$ from $v$ then $v'$ also has a unique shortest arc $\lambda'$ to $\gamma$, and moreover, $\lambda'$ is properly homotopic to the concatenation of $\lambda$ with the shortest arc of $\gamma'$ joining $v'$ to $v$: the geodesic representative of this arc has length less than $b+\epsilon/2$, whereas any other arc joining $v'$ to $\gamma$ must have length at least $b+\ell/2$, since its concatenation with the arc from $v$ to $v'$ has length at least $b+\epsilon$.

We may lift the picture to $\mathbb{H}^2$ so that $\lambda$ and $\lambda'$ form two sides of a quadrilateral $Q$ that has its other two sides in lifts $\tilde\gamma$ and $\tilde\gamma'$ of $\gamma$ and $\gamma'$, respectively, with right angles along $\gamma$. If $\alpha$ is the interior angle of $Q$ at the lift $\tilde{v}$ of $v$ then the length $a$ of $\lambda'$ is given by the formula (\ref{quad}), where $c<\epsilon/2$ is the distance from $v$ to $v'$.

But there is another arc of $\gamma'$ joining $v'$ to $v$, this one with length $\ell'-c$, where $\ell'$ is the length of $\gamma'$.  Concatenating this arc with $\lambda$ yields another proper homotopy class of arcs from $v'$ to $\gamma$, and the geodesic representative of this one has length $a'$ satisfying:\begin{align}\label{co-quad}
	\sinh a' = \sinh b\cosh (\ell'-c) + \cosh b \sinh(\ell'-c)\cos\alpha. \end{align}
This is because $a'$ is a side length of a quadrilateral $Q'$ with its other three sides in the same lifts of $\gamma$, $\gamma'$ and $\lambda$ as $Q$, intersecting $Q$ along this lift $\tilde\lambda$ of $\lambda$, such that the vertices of $Q$ and $Q'$ in the preimage of $v'$ are on opposite sides of $\tilde\lambda$.  In particular, $Q'$ has interior angle $\pi-\alpha$ at $\tilde{v}$.

Let us now vary the location of $v'$ along $\gamma'$, regarding $c$ as a parameter taking values in $[0,\ell']$ and taking $a$ and $a'$ as functions of $c$. Then by (\ref{quad}), $a$ approaches $b$ as $c\to 0$, and by (\ref{co-quad}),
\[ a'\to\sinh^{-1}\left(\sinh b\cosh \ell' + \cosh b \sinh\ell'\cos\alpha\right) \]
Our hypothesis that $\lambda$ is the unique shortest arc from $v$ to $\gamma$ implies that this quantity is larger than $b$, since it is the length of the shortest arc in the proper homotopy class of $\gamma'.\lambda$ (thinking of $\gamma'$ as a loop based at $v$).  On the other hand, as $c\to\ell'$ we have $a'\to b$, by (\ref{co-quad}), and by (\ref{quad}),
\[ a\to \sinh^{-1}\left(\sinh b\cosh \ell' - \cosh b \sinh\ell'\cos\alpha\right). \]
Again our hypotheses imply that this exceeds $b$, since it is the length of $\overline{\gamma'}.\lambda$.  Thus by the intermediate value theorem, there exists some $c_0$ for which $a(c_0)=a'(c_0)$.  It follows that $\gamma'$ intersects $\Sigma$, at the supremum of the set of $c\in[0,\ell']$ for which $\lambda'$ is the unique shortest arc from $v'$ to $\gamma$.  (The supremum in question is at most $c_0$.)

We now address the cusps of $F$.  By a standard consequence of the Margulis lemma, each rank-one cusp of $F$ has a \textit{embedded horoball neighborhood}; that is, the embedded image of $C/\langle g\rangle$, where $C$ is a horoball in $\mathbb{H}^2$ and $g$ generates a parabolic subgroup of $\pi_1 F$ fixing the ideal point of $C$.  (This is standard at least for surfaces without boundary; to see it for a surface $F$ with compact geodesic boundary, apply the standard fact to obtain an embedded horoball neighborhood $C_0/\langle g\rangle$ in $\mathit{DF}$ then pass to a sub-horoball neighborhood $C$ that is high enough to avoid the penetration of the (compact) geodesics in $\partial F$ into $C_0/\langle g\rangle$.)

Working in the upper half-plane model for $\mathbb{H}^2$, we may take $\infty$ to be the ideal point of $C$ and $g(z) = z+\tau$ for some $\tau\in(0,\infty)$ without loss of generality.  The boundary components of $\widetilde{F}$ then comprise a locally finite collection of disjoint geodesics, each the intersection with $\mathbb{H}^2$ of a circle centered in $\mathbb{R}$, with the collection invariant under the action of $\langle g\rangle$.  This combination of properties implies that among the circles containing these geodesics there is a maximum radius $r_0$; that the circles of radius $r_0$ are centered at points $x_1,\hdots,x_n$ and their translates by multiples of $\tau$, for some $n\geq 1$, where without loss of generality $x_1<\hdots<x_n<x_1+\tau$; and that all other circle radii are at most $r_0-\epsilon$ for some $\epsilon>0$.

It is not hard to see that the distance in $\mathbb{H}^2$ from any point with imaginary coordinate at least $r_0$ to one with imaginary coordinate at most $r_0-\epsilon$ is at least $\ln\left(\frac{r_0}{r_0-\epsilon}\right)$, with this bound attained by points that lie on the same vertical line.  On the other hand, a horizontal line segment of Euclidean length $\ell$ in $\mathbb{H}^2$, at height $y$, has hyperbolic length $\ell/y$.  Since every $x\in\mathbb{R}$ is within $\tau/2$ of a translate of some $x_i$ by some integer multiple of $\tau$, this implies that for some large enough $y_0$, every point $x+iy$ with $y\geq y_0$ has its nearest point on $\partial\widetilde{F}$ contained in a circle with radius $r_0$ centered in $\mathbb{R}$: one traverses a smaller distance by moving horizontally to a point directly over the center of a circle of radius $r_0$, then straight down to the apex of this circle, than by taking any path to a circle of radius at most $r_0-\epsilon$.

This observation implies that the intersection of $\widetilde{\Sigma}$ with the horoball $C_1 = \{x+iy\,|\,y\geq y_0\}$, which we may take to lie in $C$ without loss of generality, consists of points equidistant from at least two geodesics contained in circles of radius $r_0$ and centered at the $x_i$ or their $\langle g\rangle$-translates.  For geodesics $\gamma$ and $\gamma'$ contained in disjoint circles of radius $r_0$, centered at points $x$ and $x'$ in $\mathbb{R}$, the locus of points equidistant to $\gamma$ and $\gamma'$ is the hyperbolic geodesic contained in the vertical line through $(x+x')/2$.  (To see this note that the hyperbolic reflection through this geodesic, which is the restriction of the Euclidean reflection through the line containing it, exchanges $\gamma$ with $\gamma'$.) If $x = x_i$ and $x' = x_{i+1}$ for some $i<n$, or if $x = x_n$ and $x' = x_1+\tau$, then it is not hard to see that points on this vertical geodesic are closer to $\gamma$ and $\gamma'$ than to any other component of $\partial\widetilde{F}$, hence that this geodesic's intersection with $C_1$ lies in $\widetilde{\Sigma}$.  It follows as claimed in the lemma that $\Sigma\cap (C_1/\langle g\rangle)$ is a non-empty, finite disjoint union of geodesics exiting the cusp.
\end{proof}

\begin{prop}\label{quad decomp} Let $F$ be a complete, connected hyperbolic surface of finite area with compact geodesic boundary, $\gamma$ a component of $\partial F$, and $\Sigma$ the cut locus of $\gamma$.  For any component $\gamma'$ of $\partial F-\gamma$ and component $e_0$ of $\gamma'-\Sigma$, the set
\[ \bigcup_{x\in e_0} \{\lambda\,|\,\lambda\ \mbox{is a shortest arc from $x$ to}\ \gamma\} \]
is the embedded image of $\mathit{int}(Q)$ under the universal cover $p\co\widetilde{F}\to F$, for a Saccheri quadrilateral $Q$ in $\widetilde{F}$ with its base in $p^{-1}(\gamma)$ and its summit intersecting $p^{-1}(\gamma')$ in the closure of a lift of $e_0$.

The image of $\mathit{int}(Q)$ does not intersect that of $\mathit{int}(Q')$ for any quadrilateral $Q'$ corresponding to any other component $f_0$ of $(\partial F - \gamma) - \Sigma$, or that of $\mathit{int}(Q\cup\overline{Q})$ for any such union corresponding to an edge of $\Sigma$ as in Lemma \ref{quads}. The collection of all $p(Q)$ determines a decomposition of $F$ as a finite union of Saccheri quadrilaterals, each intersecting $\gamma$ exactly in its base and with its summit in $\Sigma\cup(\partial F-\gamma)$, and which intersect pairwise along edges if at all.\end{prop}

\begin{proof} The proof that a component $e_0$ of $\gamma'-\Sigma$ qives rise to a quadrilateral $Q$, for any component $\gamma'$ of $\partial F$ distinct from $\gamma$, parallels the corresponding assertion in Lemma \ref{quads}. Again for some $x\in e_0$ and a shortest arc $\lambda$ from $x$ to $\gamma$, we lift $\lambda$ to an arc $\tilde\lambda$ based at some $\tilde{x}$ in a component $\tilde{e}_0$ of $p^{-1}(e_0)$ and take $\tilde\gamma'$ and $\tilde\gamma$ to be the lifts of $\gamma'$ and $\gamma$ respectively containing $\tilde{e}_0$ and the endpoint of $\tilde\lambda$. By an open-closed argument, $\tilde\gamma$ is the closest lift of $\gamma$ to \textit{every} point of $\tilde{e}_0$. (Note that for each $\tilde{y}$ such that this holds it is the unique such lift by the definition of $e_0$, so by local finiteness of the collection of lifts of $\gamma$ there exists $\epsilon>0$ such that the same holds in an $\epsilon$-neighborhood of $\tilde{y}$ in $\tilde{e}_0$.) We then use the nearest-point retraction to $\tilde\gamma$ as in Lemma \ref{quads}.

The argument that the interior of such a quadrilateral $Q$ projects homeomorphically to $X$ and does not intersect the interior of any other embedded quadrilateral is identical to the corresponding argument for $Q\cup\overline{Q}$ in the proof of Lemma \ref{quads}.  By construction, each quadrilateral $Q$ determined by a component $e_0$ of $(\partial F-\gamma)-\Sigma$ has the closure of $e_0$ as its summit; if $Q$ is determined by an edge $e$ of $\Sigma$ then $e$ is its summit.  Again by (either) construction, the leg of each quadrilateral is the interval $\lambda$ in a geodesic ray perpendicular to $\gamma$ from some $x\in\gamma$ to its first point of intersection with $\Sigma\cup(\partial F-\gamma)$.\end{proof}

\section{The systole of loops on the three-holed sphere}\label{three hole}

The main goal of this section is to prove Theorem \ref{maximal sysloops}, characterizing the maximum value of $\mathit{sys}_p(F)$ for a three-holed sphere $F$ in terms of its boundary lengths.  But first we recall some motivation. For a fixed $k\geq 0$ and $(b_1,\hdots,b_k)\in [0,\infty)^k$, Theorem 1.2 of \cite{Gendulphe} asserts that for any complete, finite-area hyperbolic surface $F$ with $k$ geodesic boundary components of respective lengths $b_i$ (where if $b_i = 0$ the boundary component is replaced by a cusp), the value of $\mathit{sys}_p(F)$ is bounded above for all $p\in F$ by the unique positive solution $x$ to:\begin{align}\label{Gendulphe1.2}
	6(-2\chi(F)+2-k)\sin^{-1}\left(\frac{1}{2\cosh(x/2)}\right) + 2\sum_{i=1}^k \sin^{-1}\left(\frac{\cosh(b_i/2)}{\cosh(x/2)}\right) = 2\pi.
\end{align}
It further asserts that this maximum is attained at $p\in F$ if and only if the systolic loops based at $p$ divide $F$ into equilateral triangles and one-holed monogons. 

But our next result implies that for any given collection $b_1,\hdots,b_{k-1}$ of the first $k-1$ side lengths, there is an unbounded interval consisting of choices of the final side length $b_k$ which are large enough that equation (\ref{Gendulphe1.2}) lacks any positive solution. The statement of \cite[Theorem 1.2]{Gendulphe} misses this case. Theorem \ref{maximal sysloops} here fills it in for the three-holed sphere.

\begin{lemma}\label{F}  For $(b_1,\hdots,b_k)\in[0,\infty)^k$ such that $b_k = \max\{b_1,\hdots,b_k\}$, the equation (\ref{Gendulphe1.2}) has a positive solution if and only if $f_{(b_1,\hdots,b_{k-1})}(b_k)\geq \pi$, where $f_{(b_1,\hdots,b_{k-1})}$ is defined by
\[ f_{(b_1,\hdots,b_{k-1})}(x) = 6(-2\chi(F)+2-k)\sin^{-1}\left(\frac{1}{2\cosh(x/2)}\right) + 2\sum_{i=1}^{k-1} \sin^{-1}\left(\frac{\cosh(b_i/2)}{\cosh(x/2)}\right). \]
This is a continuous and decreasing function of $x$ that limits to $0$ as $x\to\infty$.\end{lemma}

\begin{proof} The function $f \doteq f_{(b_1,\hdots,b_{k-1})}$ records the left-hand side of the formula (\ref{Gendulphe1.2}) but excludes its final term: the summand $\sin^{-1}\left(\frac{\cosh(b_k/2)}{\cosh(x,2)}\right)$. We begin by observing that formula (\ref{Gendulphe1.2}) only makes sense for (positive) $x\in [b_k,\infty)$ (recalling that $b_k = \max_i\{b_i\}$). This is because the domain of the arcsine is the range of the sine function, $[-1,1]$, and if $0\leq x<b_k$ then $\cosh(b_k/2)/\cosh(x/2)>1$.

We next note that both $f$ and the left side of formula (\ref{Gendulphe1.2}), regarded as a function of $x$, are decreasing on $[b_k,\infty)$. This follows from the facts that $\cosh(x/2)$ and the inverse sine are increasing functions, and that the only $x$-dependence in the definitions of $f$ and (\ref{Gendulphe1.2}) is the appearance of $\cosh(x/2)$ in the denominator of each summand. Therefore if equation (\ref{Gendulphe1.2}) has a solution then its left-hand side's value at $x=b_k$ is at least $2\pi$. This is equivalent to $f(b_k) = \pi$ because the final term of the left side of (\ref{Gendulphe1.2}) is $\sin^{-1}\left(\frac{\cosh(b_k/2)}{\cosh(b_k/2)}\right) = \pi/2$ for $x = b_k$.  

On the other hand, $f(x)$ limits to $0$ as $x\to\infty$, since $\sin^{-1}(0) = 0$ and the inverse sine is continuous. By the intermediate value theorem, this implies that if $f(b_k) \geq \pi$ then there exists $x\in [b_k,\infty)$ (unique, since $f$ is decreasing) such that $f(b_k)=\pi$ and hence (\ref{Gendulphe1.2}) holds.\end{proof}


In the rest of this section, ``$F$'' always refers to a three-holed sphere with geodesic boundary. For each such $F$ it is well known that there is a unique shortest arc between any two components of $\partial F$, meeting both of them at right angles, and that the collection of all three such arcs cuts $F$ into a pair of isometric right-angled hexagons exchanged by a reflective involution of $F$.  (Establishing this is a key step in the description of ``Fenchel--Nielsen coordinates'' for the Teichm\"uller spaces of arbitrary hyperbolic surfaces, see eg.~\cite[\S 10.6]{FaMa}.) 

Below we will denote the geodesic boundary components of $F$ as $B_1$, $B_2$, and $B_3$ and their lengths as $b_1$,  $b_2$, and $b_3$, respectively. Then the resulting right-angled hexagons have alternate side lengths $\frac{b_1}{2}, \frac{b_2}{2}$, and $\frac{b_3}{2}$, which determine them up to isometry (see eg.~\cite[Theorem 3.5.14]{Ratcliffe}), and it follows that $F$ itself is determined up to isometry by the triple $(b_1,b_2,b_3)$.

We now begin the process of understanding $\mathit{sys}_p(F)$ for fixed $(b_1,b_2,b_3)$.

\begin{lemma}\label{simple loop} For any $p\in F$, $\mathit{sys}_p(F) = \min\{\ell_1,\ell_2,\ell_3\}$, where $\ell_i$ is the length of a simple loop freely homotopic to $B_i$ for each $i\in\{1,2,3\}$. 

Writing $F = H\cup\bar{H}$ for isometric right-angled hexagons $H$ and $\bar{H}$ such that $H\cap\bar{H}$ is a disjoint union of three edges, if $p$ lies in an edge of $H\cap\bar{H}$ that has endpoints in $B_i$ and $B_j$, then $\mathit{sys}_p(F) = \min\{\ell_i,\ell_j\}$. In fact, for such a point $p$ and any loop $\gamma$ based at $p$ which is not freely homotopic to a power of either $B_i$ or $B_j$, the length of $\gamma$ is strictly greater than $\max\{\ell_i,\ell_j\}$.\end{lemma}

\begin{proof} Fix $p\in F$. By standard results there is a unique geodesic loop (parametrized proportional to arclength, with parameter domain $[0,1]$) in the based homotopy class of every homotopically nontrivial loop based at $p$, the \textit{systole at $p$}, which minimizes length in this homotopy class. That is, the systole at $p$ has length $\mathit{sys}_p(F)$.

As the referee has pointed out, the Lemma's first assertion follows quickly from the well known facts that every simple closed curve on the three-holed sphere is freely homotopic to some $B_i$, and that closed geodesics such as the $B_i$ minimize length in their free homotopy classes. By a standard surgery argument the systole at any $p$ is a simple loop, so its length is at least that of the $B_i$ to which it is homotopic. Below we give a more elaborate surgery argument then follow the same strategy to prove the Lemma's further assertions.

Let $\gamma\co [0,1]\to F$ be a loop based at $p$ of minimal length over all non-trivial homotopy classes; i.e.~with length $\ell = \mathit{sys}_p(F)$. Partition $[0,1]$ using $\gamma$'s intersections with $H$ and $\bar{H}$, into:
\[ 0 = t_0 < t_1 < \hdots < t_n = 1, \]
such that for each $i>0$, $\gamma|_{[t_{i-1},t_i]}$ maps into either $H$ or $\bar{H}$, and if $\gamma|_{[t_{i-1},t_i]}$ maps into $H$ then $\gamma|_{[t_{i},t_{i+1}]}$ maps into $\bar{H}$ and vice-versa. We observe that $n>1$ since neither $H$ nor $\bar{H}$ contains a closed geodesic loop, each being isometric to a convex subset of $\mathbb{H}^2$.

Assume first that $p\in H$ does not lie in one of the three arcs of $H\cap\bar{H}$. In this case $n\geq 3$ since both $\gamma|_{[t_0,t_1]}$ and $\gamma|_{[t_{n-1},t_n]}$ map into $H$. We claim that in fact $n=3$; i.e., only the single segment $\gamma|_{[t_1,t_2]}$ of $\gamma$ maps into $\bar{H}$. Before proving this claim we note that in this case, $\gamma$ is simple. And $\gamma|_{[t_1,t_2]}$ has its endpoints on different components of $H\cap\bar{H}$, since it is geodesic and $\bar{H}$ is isometric to a right-angled hexagon in $\mathbb{H}^2$. These components of $H\cap\bar{H}$ have endpoints on a common component $B_i$ of $\partial F$, and it is easy to see directly that $\gamma$ is freely homotopic to $B_i$.

We prove the claim by contradiction. Assuming that $n>3$, for any $k \geq 2$ such that $\gamma|_{[t_k, t_{k+1}]}$ lies in $\bar{H}$, replace the segment with its mirror image $\bar{\gamma}|_{[t_k,t_{k+1}]}$ in $H$.
 Since the reflection exchanging $H$ with $\bar{H}$ leaves $H\cap \bar{H}$ fixed, we thus obtain a continuous broken geodesic
\[ \gamma|_{[t_0,t_1]}.\gamma|_{[t_1,t_2]}.\left(\gamma|_{[t_2,t_3]}.\bar{\gamma}|_{[t_3, t_4]}.\cdots .\gamma|_{[t_{n-1},t_n]}\right) \]
based at $p$, with the same length as $\gamma$.  
Now the entire broken geodesic within  the parentheses above maps into $H$, which as we previously mentioned is isometric to a convex subset of $\mathbb{H}^2$, so it can be replaced in the concatenation by the strictly shorter geodesic arc in $H$ that joins its endpoints.  Arguing as in the $n=3$ case above shows that the resulting simple loop at $p$ is freely homotopic to a boundary component and hence not null-homotopic, contradicting our hypothesis that $\ell = \mathit{sys}_p(F)$.

The case that $p\in \bar{H} - (H\cap\bar{H})$ is completely analogous, so let us now take $p\in H\cap\bar{H}$. In this case we claim that $n=2$, i.e.~that $\gamma = \gamma|_{[t_0,t_1]}.\gamma|_{[t_1,t_2]}$. If this is so then clearly $\gamma$ is simple and freely homotopic to the boundary component containing the endpoints of the edges of $H \cap\bar{H}$ containing $\gamma(t_0) = p = \gamma(t_2)$ and $\gamma(t_1)$. The claim's proof is similar to the one above: for each odd $k>2$, replacing $\gamma|_{[t_{i-1},t_i]}$ by its reflection $\bar\gamma|_{[t_{i-1},t_i]}$ across $H\cap\bar{H}$ yields the broken geodesic
\[ \gamma|_{[t_0,t_1]}.\left(\gamma|_{[t_1,t_2]}.\bar\gamma|_{[t_2,t_3]}.\cdots .\bar\gamma|_{[t_{n-1},t_n]}\right), \]
a loop based at $p$ with the same length as $\gamma$ but such that the entire arc in parentheses maps into one of $H$ or $\bar{H}$. Replacing that arc by the geodesic arc between its endpoints in the same hexagon produces a shorter but still homotopically non-trivial loop based at $p$, a contradiction.

Now let us continue with $p\in H\cap \bar{H}$, say in the edge from $B_i$ to $B_j$, and let $\gamma\co[0,1]\to F$ be an arbitrary geodesic loop based at $p$, of (not necessarily minimal) length $\ell$, that is not freely homotopic to a power of $B_i$ or $B_j$ for distinct $i,j\in\{1,2,3\}$. Then $\gamma$ intersects both the edge $e_i$ of $H\cap\bar{H}$ opposite $B_i$, since $F-e_i$ deformation retracts to $B_i$, and $\gamma$ similarly intersects the edge $e_j$ opposite $B_j$. Upon decomposing $\gamma$ into segments as above, and taking $k$ to be minimal such that $\gamma(t_k)\in e_i$,we can replace $\gamma$ by the broken geodesic loop
\[ \tilde\gamma = (\gamma|_{[t_0,t_1]}.\tilde\gamma|_{[t_1,t_2]}.\cdots.\tilde\gamma|_{[t_{k-1}t_k]}).(\tilde\gamma|_{[t_kt_{k+1}]}.\cdots.\tilde\gamma|_{[t_{n-1}t_n]}), \]
where $\tilde\gamma|_{[t_{i-1}t_i]} = \gamma|_{[t_{i-1}t_i]}$ or $\tilde\gamma|_{[t_{i-1}t_i]} = \bar\gamma|_{[t_{i-1}t_i]}$ for each $i$, chosen so that each segment inside the left-hand parentheses above lies in the same hexagon (out of $H$ or $\bar{H}$) and each segment in the right-hand group lies in the other. At least one of the parentheses contains at least two segments, since $\gamma$ also intersects $e_j$. Therefore the geodesic representative of $\tilde\gamma$ is strictly shorter than $\tilde\gamma$, which has length $\ell$. This geodesic representative has length $\ell_j$, since it intersects $e_i'$ once, so $\ell>\ell_j$. Redefining $k$ as the minimal index such that $\gamma(t_k)\in e_j$ and running the analogous argument shows that $\ell>\ell_i$ so we have $\ell>\max\{\ell_i,\ell_j\}$ for such $\gamma$.\end{proof}

\begin{figure}
    \centering
        \includegraphics[scale=0.75]{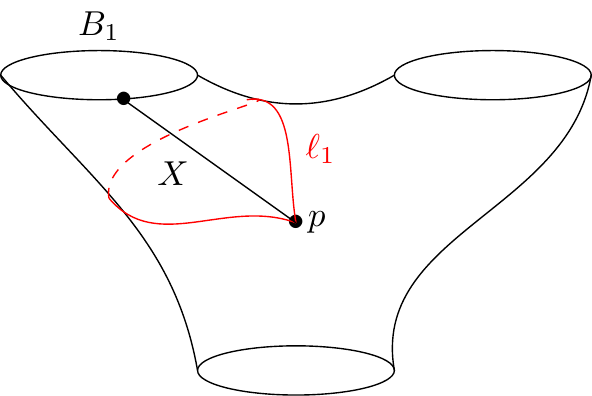}
        \caption{}
        \label{fig:inside}
\end{figure}

\begin{lemma}\label{loop length} For any $p\in F$ and any boundary component $B_i$, the simple loop based at $p$ and freely homotopic to $B_i$ has length $\ell_i$ determined by
	$$\sinh{\left(\frac{\ell_i}{2}\right)} = \cosh{X_i}\sinh{\left(\frac{b_i}{2}\right)}, $$
where $X_i$ is the length of the shortest geodesic arc in $F$ from $p$ to $B_i$. (See Figure \ref{fig:inside}.) The angle $\delta_i$ between the loop and the arc satisfies $\sin(\delta_i) = \frac{\cosh(b_i/2)}{\cosh(\ell_i/2)}$. \end{lemma}

\begin{proof} Consider the one-holed monogon bounded by a simple geodesic loop based at $p$ and freely homotopic to $B_i$.  
Cutting it along the geodesic arc connecting $p$ to $B_i$ results in a Saccheri quadrilateral with legs of length $X_i$, base of length $b_i$, and summit of length $\ell_i$. Since this quadrilateral's legs have equal length, the arc joining the midpoint of its base to that of its summit divides it into isometric sub-quadrilaterals exchanged by a reflection, each with three right angles. Standard hyperbolic trigonometric identities for such quadrilaterals now give the formulas claimed, compare eg.~\cite[VI.3.3]{Fenchel}.\end{proof}


\begin{lemma}\label{inside} Write $F = H\cup\bar{H}$ for isometric right-angled hexagons $H$ and $\bar{H}$ such that $H\cap\bar{H}$ is a disjoint union of three edges.  If $\mathit{sys}_p(F)$ attains a maximum at $p\in\mathit{int}(H)$ or $\mathit{int}(\bar{H})$, then the simple loops based at $p$ and freely homotopic to each boundary component all have equal lengths.
\end{lemma}

\begin{proof}
By Lemma \ref{loop length}, the distance of $p$ from a boundary component and the length of the loop based at $p$ around that component increase or decrease together.  Supposing that the loop about $B_3$ gives the systolic value with the loop about $B_1$ strictly longer, i.e. $\ell_3\leq \ell_2,$ and $\ell_3 < \ell_1$, the idea is to move $p$ further from both $B_2$ and $B_3$ simultaneously, increasing the systolic loop length and yielding a contradiction.

It is easy to see how to do this in the hyperboloid model for $\mathbb{H}^2$, contained in the Lorentzian space $\mathbb{R}^{1,2}$ as described in Chapter 3 of \cite{Ratcliffe}. To this end, assume that $p\in\mathit{int}(H)$ and regard $H$ as a subset of $\mathbb{H}^2$. Since $H$ is right-angled, the shortest arcs from $p$ to each of the $B_i$ all lie in $H$ as well, so we may take the whole picture to lie in $\mathbb{H}^2$. In particular, for each $i\in\{1,2,3\}$ the edge of $H$ that was contained in $B_i$ --- which we will simply call $B_i$ here, see Figure \ref{fig:hexagon labels} --- now lies in a hyperbolic geodesic that is itself the intersection of a two-dimensional time-like subspace of $\mathbb{R}^{1,2}$ with $\mathbb{H}^2$. (See the Definition above Theorem 3.2.6 in \cite{Ratcliffe}.) 

For each $i\in\{1,2,3\}$, let $x_i$ be the unit space-like normal in $\mathbb{R}^{1,2}$ to the subspace containing $B_i$ that has the property that the Lorentzian inner product $p\circ x_i$ is less than $0$. (Every time-like subspace of codimension one in $\mathbb{R}^{1,n}$ has a space-like normal vector, cf.~\cite[Lemma 1.1]{DeB_Delaunay}.) By \cite[Theorem 3.2.8]{Ratcliffe}, the distance $X_i$ from $p$ to $B_i$ satisfies $\sinh X_i = - p\circ x_i$.

The tangent space to $\mathbb{H}^2$ at $p$ is $p^{\perp} = \left\{v\in\mathbb{R}^{1,2}\,|\,v\circ p = 0\right\}$, a space-like subspace by \cite[Theorem 3.1.5]{Ratcliffe}. A unit vector $v\in p^{\perp}$ determines a geodesic $\lambda\co\mathbb{R}\to\mathbb{H}^2$ parametrized by arclength, given by $\lambda(t) = (\cosh t) p + (\sinh t) v$, with $\lambda(0) = p$ and $\lambda'(0) = v$.  Thus by the paragraph above, in order to ensure that the distance between $\lambda(t)$ and $B_i$ increases for $i=2,3$, we must make $\lambda(t)\circ x_i$ \textit{decrease}, for instance by choosing $v$ so that $v\circ x_2<0$ and $v\circ x_3<0$. In fact, since $\lambda''(0) = p$ pairs negatively with $x_2$ and $x_3$, it is enough to have $v\circ x_i\leq 0$ for $i=2,3$.

The set of all vectors $v$ with $v\circ x_2\leq 0$ is a closed half-space of $\mathbb{R}^{1,2}$ that intersects $p^{\perp}$ in a closed half-plane (since $p^{\perp}\ne x_2^{\perp}$). Similarly, the set $v\circ x_3\leq 0$ is a closed half-plane in $p^{\perp}$, and the result now follows from the standard exercise that in a two-dimensional vector space, the intersection of any two closed half-planes contains a line.\end{proof}

\begin{figure}
    \centering
    \begin{subfigure}[b]{0.3\textwidth}
        \includegraphics[scale=0.75]{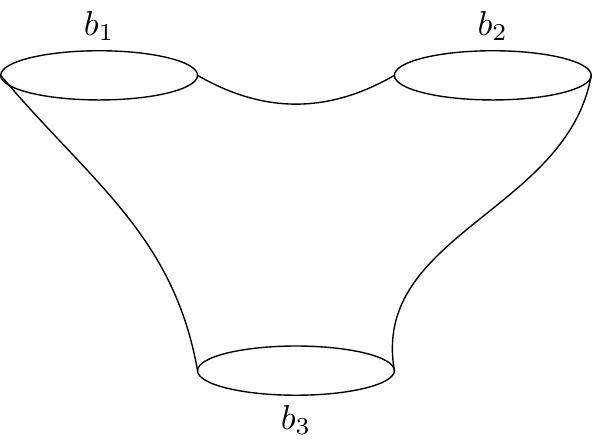}
        \caption{The three-holed sphere}
        \label{fig:3hole}
    \end{subfigure}
    ~ 
    \begin{subfigure}[b]{0.4\textwidth}
        \includegraphics[scale=0.75]{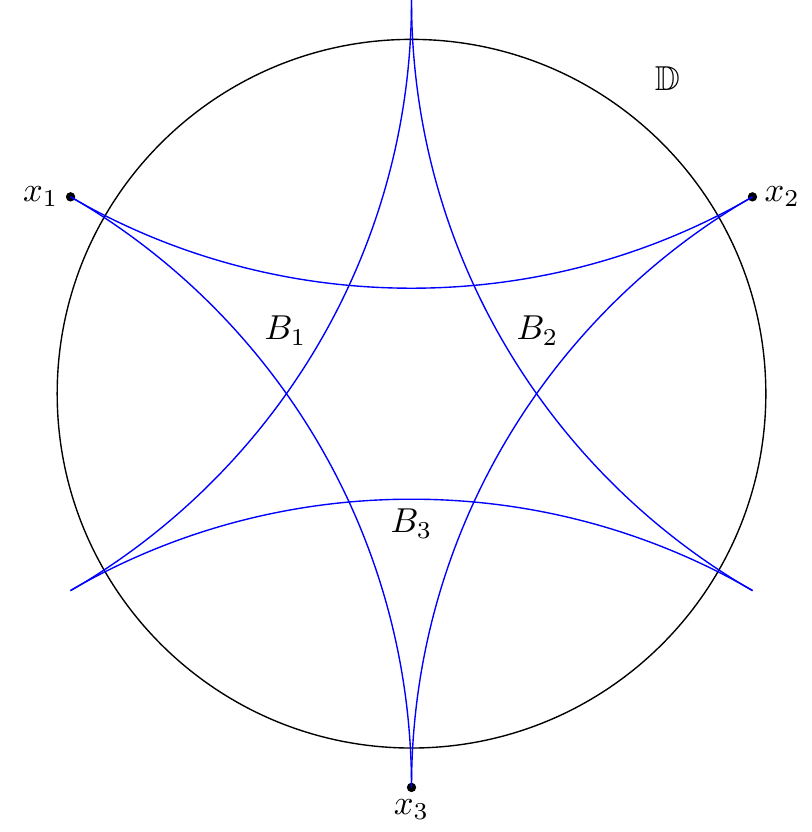}
        \caption{Hexagon in Lorentzian space}
        \label{fig:hexagon}
    \end{subfigure}
    \caption{ }
    \label{fig:hexagon labels}
\end{figure}

The hypotheses of Lemma \ref{inside} require $p$ to lie in the \textit{interior} of $H$ or $\bar{H}$ because its proof requires $p$ to move in an arbitrary direction and remain in $H$ or $\bar{H}$, so that we can continue to interpret $\mathit{sys}_p(F)$ only in terms of the hexagon. We now consider cases when $p$ lies in an edge.

\begin{lemma}\label{interior} Writing $F = H\cup\bar{H}$ for isometric right-angled hexagons $H$ and $\bar{H}$ such that $H\cap\bar{H}$ is a disjoint union of three edges, the maximum of $\mathit{sys}_p(F)$ is not attained at any point $p$ in the interior of an edge of $H\cap\bar{H}$.
\end{lemma}

\begin{proof}
Say $p$ lies on the edge $B_1'$ of $H\cap\bar{H}$ joining $B_2$ to $B_3$. By Lemma \ref{simple loop}, $\mathit{sys}_p(F) = \min\{\ell_2,\ell_3\}$. We prove the current result by showing that the lengths $\ell_2$ and $\ell_3$ of the simple loops based at $p$ and freely homotopic to $B_2$ and $B_3$, respectively, are simultaneously increased by moving $p$ off of the geodesic $B_1'$ perpendicularly into $H$. We lift $H$ to $\mathbb{H}^2$ and use notation from the proof of Lemma \ref{inside}, in particular taking $x_2$ and $x_3$ to be space-like unit normals to the geodesics containing the edges of $H$ that lie in $B_2$ and $B_3$, respectively. Since by hypothesis $p$ lies in the interior of $B_1'$ --- that is, not in $B_2$ or $B_3$ --- we may again choose $x_2$ and $x_3$ so that $p\circ x_2<0$ and $p\circ x_3<0$.

The proof of Theorem 3.2.7 of \cite{Ratcliffe} implies that the hyperbolic geodesic $N$ containing $B_1'$ is the intersection between $\mathbb{H}^2$ and the span of $x_2$ and $x_3$ in $\mathbb{R}^{1,2}$. Therefore a vector $v\in p^{\perp}$ and normal to $N$ is normal to each of $x_2$ and $x_3$. Taking $v$ to point into $H$ and arguing as in the proof of Lemma \ref{inside} shows that moving $p$ along the geodesic $\lambda$ with $\lambda(0) = p$ and $\lambda'(0) = v$ thus increases its distances to $B_2$ and $B_3$, hence also $\ell_2$ and $\ell_3$ by Lemma \ref{loop length}.
\end{proof}


\begin{lemma}\label{boundary} If $\mathit{sys}_p(F)$ attains a maximum at $p\in\partial F$, say $p\in B_3$ without loss of generality, then $p$ is not the endpoint of a shortest geodesic arc joining $B_3$ to $B_1$ or $B_2$, and 
\[ \mathit{sys}_p(F) = \ell_1 = \ell_2 \leq b_3, \] 
where $\ell_i$ is the length of the simple loop based at $p$ and freely homotopic to $B_i$. In particular, $B_3$ is the unique longest boundary component of $F$.
\end{lemma}

\begin{proof}
We first observe that if $\mathit{sys}_p(F)$ attains a maximum at $p\in B_3$ then $\ell_3 = b_3 \geq \min\{\ell_1,\ell_2\}$. For if $\ell_3$ is less than both $\ell_1$ and $\ell_2$ then $\mathit{sys}_p(F) = \ell_3$ has this property as well. But $\ell_3$, and hence in this case also $\mathit{sys}_p(F)$, can be increased by just moving $p$ off of $B_3$ into $F$ (compare Lemma \ref{loop length}). So if $\mathit{sys}_p(F)$ attains a maximum at $p\in B_3$ then $\mathit{sys}_p(F) = \min\{\ell_1,\ell_2\}\leq b_3$.

Now let $p_i$ be the endpoint on $B_3$ of the shortest arc joining it to $B_i$, for $i=1,2$. These arcs comprise two of the three edges of the intersection $H\cap \bar{H}$, where (as before) $H$ and $\bar{H}$ are right-angled hexagons such that $F=H\cup\bar{H}$ and $H\cap\bar{H}$ is a union of edges. It is a fundamental hyperbolic trigonometric fact that the distance to $B_i$ increases in both $H\cap B_3$ and $\bar{H}\cap B_3$ as one moves away from $p_i$. (It can be showed by an exercise using the notation from the proof of Lemma \ref{inside}, for instance.) 

It follows that the only local maximum of $\ell_1$ on $B_3$ occurs at $p=p_2$, the furthest point of $B_3$ from $p_1$. Therefore at a maximum of $\mathit{sys}_p(F)$ on $B_3 - \{p_2\}$ we must have $\ell_1 \geq \ell_2$, since otherwise $\mathit{sys}_p(F) = \ell_1$ could be increased by moving $p$. By the same token, the only local maximum of $\ell_2$ is at $p=p_1$ and we must have $\ell_2\geq\ell_1$ at a maximum of $\mathit{sys}_p(F)$ on $B_3-\{p_1\}$.  Thus if the maximum occurs at $p\in B_2 - \{p_1,p_2\}$ then $\mathit{sys}_p(F) = \ell_1 = \ell_2$.

If $\mathit{sys}_p(F)$ attains a maximum at $p=p_1$, say, then by the above we would have $\mathit{sys}_{p_1}(F) = \ell_2\leq \ell_1$. But Lemma \ref{simple loop} implies that $\ell_2 >\ell_1$ at $p_1$, a contradiction, since no loop freely homotopic to $B_2$ is also freely homotopic to a power of either $B_1$ or $B_3$. Similarly, $\mathit{sys}_p(F)$ also does not attain a maximum at $p = p_2$. Therefore by the above, if $\mathit{sys}_p(F)$ attains a maximum at $p\in B_3$ then at this point $p$, $\mathit{sys}_p(F) =\ell_1 = \ell_2\leq b_3$. Since $p$ has positive distance from each of the other boundary components $B_1$ and $B_2$, here we have $b_1 < \ell_1$ and $b_2 < \ell_2$, and it follows that $B_3$ is the unique longest boundary component.\end{proof}

\begin{figure}
    \centering
        \includegraphics[scale=0.75]{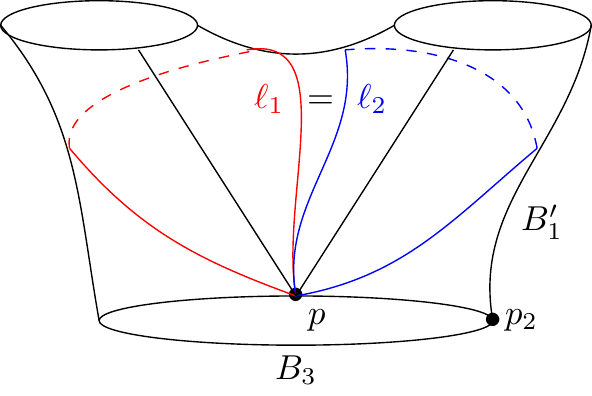}
        \caption{}
        \label{fig:boundary}
\end{figure}

Lemma \ref{boundary} notably does \textit{not} assert that the simple loops freely homotopic to the different boundary components all have the same length at a maximizer for $\mathit{sys}_p(F)$ on $\partial F$. Indeed, Lemma \ref{F} implies that this \textit{cannot} occur for certain triples $(b_1,b_2,b_3)$ of boundary component lengths. This accounts for the difference between \cite[Th.~1.2]{Gendulphe} and the corrected version below.

\begin{thm}\label{maximal sysloops}\MaxSysLoops
\end{thm}

\begin{proof} By Lemma \ref{simple loop} the function $p\mapsto\mathit{sys}_p(F)$ is the minimum of functions $\ell_1$, $\ell_2$, $\ell_3$ each of which depends continuously on $p$ by the formula given in Lemma \ref{loop length}. Therefore since $F$ is compact, $p\mapsto \mathit{sys}_p(F)$ attains a maximum at some $p\in F$.

If the maximum is attained at a point $p\in\mathit{int}(F)$ then by Lemma \ref{inside}, for each $i$ the simple geodesic loop based at $p$ and freely homotopic to $B_i$ has length equal to $\mathit{sys}_p(F)$; i.e.~it is a systolic loop. Each systolic loop bounds an annulus (i.e.~a one-holed monogon) with the boundary component it encircles, and the loops do not intersect except at $p$: if they did then this would give rise to a bigon in $F$ bounded by geodesic arcs, which is impossible. Therefore the monogons bounded by the three systolic loops meet only at $p$. By inspection, the closure of their complement is the image of an equilateral triangle in $\mathbb{H}^2$, embedded in $F$ except that its vertices are all identified to $p$, with its edges sent to the systolic loops.  The vertex angles of the monogons and the equilateral triangle therefore sum to $2\pi$, since these all meet at $p$.

As an equilateral triangle with side length $2r$ has vertex angle $\alpha(r)$ given in Theorem \ref{maximal injrad}, and a one-holed monogon with hole radius $b_i$ has vertex angle $\delta_i$ given in Lemma \ref{loop length}, we find that $f_{(b_1,b_2)}(x)$ described above records the vertex angle sum of an equilateral triangle with one-holed monogons of hole radii $b_1$ and $b_2$, all with side length $x$. Therefore $x = \mathit{sys}_p(F)$ satisfies
\begin{align}\label{eq1} f_{(b_1,b_2)}(x) + 2\sin^{-1}\left(\frac{\cosh(b_3/2)}{\cosh(x/2)}\right) = 2\pi \end{align}
as claimed. Since $f_{(b_1,b_2)}$ is the three-holed sphere case of the function of Lemma \ref{F}, that Lemma's conclusion implies that $f_{(b_1,b_2)}(b_3) \geq \pi$. But we claim that in fact strict inequality holds in this case. To this end, note that $p\in\mathit{int}(F)$ has positive distance to each boundary component $B_i$ and hence $x= \ell_i >b_i$ for each $i$ by Lemma \ref{loop length}.  So since $x>b_3$ the inverse sine term above has value less than $\pi$ and hence $f_{(b_1,b_2)}(x) >\pi$. Since $f_{(b_1,b_2)}(x)$ decreases with $x$, the claim follows.

On the other hand, if $f_{(b_1,b_2)}(b_3) > \pi$ then by the other direction of Lemma \ref{F}'s conclusion there exists a unique $x > b_3$ for which equation (\ref{eq1}) holds. For this $x$, we may identify the edges of an equilateral triangle to those of monogons with hole radii $b_1$, $b_2$, and $b_3$, all with side length $x$ to produce an identification space $F$ homeomorphic to the three-holed sphere. This identification space further inherits a hyperbolic structure from its constituent pieces. In particular, because equation (\ref{eq1}) is satisfied a standard argument shows that the vertex quotient $p\in F$ has a neighborhood isometric to an open disk in $\mathbb{H}^2$ that is a union of wedges: three of angle $\alpha(x/2)$ and one each of angle $\delta_i$ for $i = 1,2,3$. By construction, $F$ has boundary components $B_1$, $B_2$ and $B_3$ of lengths $b_1$, $b_2$ and $b_3$, respectively, and the simple loops based at $p$ and freely homotopic to each $B_i$ all have length $x$. Therefore by Lemma \ref{loop length}, $\mathit{sys}_p(F) = x$.

As noted above for this $x$ we have $x > b_3\geq b_2\geq b_1$. On the other hand, for any point $p'$ in a boundary component $B_i$, $\mathit{sys}_{p'}(F) \leq \ell_i = b_i$. Therefore for the three-holed sphere $F$ constructed above, the function $p\mapsto\mathit{sys}_p(F)$ attains its maximum in the interior of $F$ --- in fact at $p$, by what was already showed --- and we have established the Theorem's assertion (a).

Assertion (a) implies that $\mathit{sys}_p(F)$ attains its maximum at $p\in\partial F$ if and only if $f_{(b_1,b_2)}(b_3)\leq\pi$. Let us henceforth assume that this is the case, and note that it implies that $b_3$ is strictly greater than $b_2\geq b_1$. This is because the summand $2\sin^{-1}\left(\frac{\cosh(b_2/2)}{\cosh(x/2)}\right)$ takes the value $\pi$ at $x = b_2$, so $f_{(b_1,b_2)}(b_2) > \pi$.  Lemma \ref{boundary} asserts that the point $p$ at which the maximum is attained lies on the longest boundary component, that is, $B_3$, and that $\mathit{sys}_p(F) = \ell_1 = \ell_2$ at this point.

As argued in the proof of part (a), the systolic loops based at $p$ and encircling $B_1$ and $B_2$ each bound a one-holed monogon with their respective boundary components. Inspecting Figure \ref{fig:boundary}, one finds that the complement of the union of these monogons' interiors is a hyperbolic isosceles triangle with two sides of length $\mathit{sys}_p(F)$ and one of length $b_3$, and all its vertices identified to $p$. Setting $x = \mathit{sys}_p(F)$, hyperbolic trigonometric calculations give the following formulas for the two equal angles $\theta$ and the other angle $\psi$ of such a triangle:
\begin{align*}
	& \cos(\psi/2) = \frac{\tanh(b_3/2)}{\tanh(x)} & & \sin(\theta) = \frac{\sinh(b_3/2)}{\sinh(x)} \end{align*}
Since $p$ lies on $\partial F$, the total vertex angle sum of the one-holed monogons bounded by $B_1$ and $B_2$ and the isosceles triangle in their complement is $\pi$. This gives the equation $g(x) = \pi$, for $g$ as described in the theorem statement. Of course $x=\ell_2>b_2$ since $p$ has positive distance to $B_2$, and as observed in Lemma \ref{boundary}, $x\leq b_3$. For another way of seeing this last point we can note that $g$ is a decreasing function of $x$, that $g(b_2)>\pi$ since it shares the summand of $f_{(b_1,b_2)}(b_2)$ described above, and that $g(b_3) = f(b_3) \leq \pi$. (This requires some trignometric manipulation.) This completes the theorem's proof.\end{proof}

\bibliographystyle{plain}
\bibliography{3HoleSphere}

\begin{thebibliography}{10}

\bibitem{Bavard_anneaux}
C.~Bavard.
\newblock Anneaux extr\'{e}maux dans les surfaces de {R}iemann.
\newblock {\em Manuscripta Math.}, 117(3):265--271, 2005.

\bibitem{Bavard}
Christophe Bavard.
\newblock Disques extr\'{e}maux et surfaces modulaires.
\newblock {\em Ann. Fac. Sci. Toulouse Math. (6)}, 5(2):191--202, 1996.

\bibitem{BowEp}
B.~H. Bowditch and D.~B.~A. Epstein.
\newblock Natural triangulations associated to a surface.
\newblock {\em Topology}, 27(1):91--117, 1988.

\bibitem{DeB_Voronoi}
Jason DeBlois.
\newblock The centered dual and the maximal injectivity radius of hyperbolic
  surfaces.
\newblock {\em Geom. Topol.}, 19(2):953--1014, 2015.

\bibitem{DeB_cyclic}
Jason DeBlois.
\newblock The geometry of cyclic hyperbolic polygons.
\newblock {\em Rocky Mountain J. Math.}, 46(3):801--862, 2016.

\bibitem{DeB_Delaunay}
Jason DeBlois.
\newblock The {D}elaunay tessellation in hyperbolic space.
\newblock {\em Math. Proc. Cambridge Philos. Soc.}, 164(1):15--46, 2018.

\bibitem{DeB_many}
Jason DeBlois.
\newblock Bounds for several-disk packings of hyperbolic surfaces.
\newblock {\em J. Topol. Anal.}, 12(1):131--167, 2020.

\bibitem{Fanoni}
Federica Fanoni.
\newblock The maximum injectivity radius of hyperbolic orbifolds.
\newblock {\em Geom. Dedicata}, 175:281--307, 2015.

\bibitem{FaMa}
Benson Farb and Dan Margalit.
\newblock {\em A primer on mapping class groups}, volume~49 of {\em Princeton
  Mathematical Series}.
\newblock Princeton University Press, Princeton, NJ, 2012.

\bibitem{FejesToth}
L.~Fejes~T\'{o}th.
\newblock Kreisausf\"{u}llungen der hyperbolischen {E}bene.
\newblock {\em Acta Math. Acad. Sci. Hungar.}, 4:103--110, 1953.

\bibitem{Fenchel}
Werner Fenchel.
\newblock {\em Elementary geometry in hyperbolic space}, volume~11 of {\em De
  Gruyter Studies in Mathematics}.
\newblock Walter de Gruyter \& Co., Berlin, 1989.
\newblock With an editorial by Heinz Bauer.

\bibitem{Gendulphe}
Mathieu Gendulphe.
\newblock The injectivity radius of hyperbolic surfaces and some {M}orse
  functions over moduli spaces.
\newblock Preprint. arXiv:1510.02581, October 2015.

\bibitem{Gendulphe_systole}
Matthieu Gendulphe.
\newblock Systole et rayon interne des vari\'{e}t\'{e}s hyperboliques non
  compactes.
\newblock {\em Geom. Topol.}, 19(4):2039--2080, 2015.

\bibitem{Kojima}
Sadayoshi Kojima.
\newblock Polyhedral decomposition of hyperbolic {$3$}-manifolds with totally
  geodesic boundary.
\newblock In {\em Aspects of low-dimensional manifolds}, volume~20 of {\em Adv.
  Stud. Pure Math.}, pages 93--112. Kinokuniya, Tokyo, 1992.

\bibitem{Ratcliffe}
John~G. Ratcliffe.
\newblock {\em Foundations of hyperbolic manifolds}, volume 149 of {\em
  Graduate Texts in Mathematics}.
\newblock Springer, New York, second edition, 2006.

\bibitem{Schmutz_maxRiem}
Paul Schmutz.
\newblock Congruence subgroups and maximal {R}iemann surfaces.
\newblock {\em J. Geom. Anal.}, 4(2):207--218, 1994.

\end{thebibliography}

\end{document}